\documentclass[12pt,reqno]{amsart}
\usepackage{amsmath}
\usepackage{amssymb}
\usepackage{verbatim}
\usepackage{mathrsfs}
\usepackage{graphicx} 
\usepackage{caption}
\usepackage{subcaption}
\usepackage[noadjust]{cite}
\usepackage{hyperref}
\usepackage{color}
\usepackage[bottom]{footmisc}

\usepackage[top=0.5in,bottom=0.7in,left=0.67in,right=0.67in]{geometry}

\newtheorem{lemma}{Lemma}
\newtheorem{prop}[lemma]{Proposition}
\newtheorem{cor}[lemma]{Corollary}
\newtheorem{theorem}[lemma]{Theorem}

\newtheorem{exmp}{Example}

\newtheorem{definition}[lemma]{Definition}

\numberwithin{equation}{section}
\numberwithin{lemma}{section}

\newcommand{\C}{\mathbb{C}}    
\newcommand{\N}{\mathbb{N}}    
\newcommand{\PL}{\mathbb{P}}   
\newcommand{\R}{\mathbb{R}}    
\newcommand{\Z}{\mathbb{Z}}    

\newcommand{\wh}{\widehat}
\renewcommand{\le}{\leqslant}
\renewcommand{\ge}{\geqslant}
\newcommand{\bs}{\backslash}
\newcommand{\ol}{\overline}
\newcommand{\la}{\langle}
\newcommand{\ra}{\rangle}
\newcommand{\bo}{\mathscr{O}} 

\newcommand{\eps}{\epsilon}
\newcommand{\pp}{\mathsf{p}}

\newcommand{\DG}{{\mathsf{Diag}}}

\DeclareMathOperator{\FF}{\mathsf{F}}

\DeclareMathOperator{\UU}{\pmb{U}}
\newcommand{\er}{\eqref}

\newcommand{\mra}{{\mathring{a}}}
\newcommand{\mrb}{{\mathring{b}}}

\newcommand{\mrv}{{\mathring{v}}}
\newcommand{\mrw}{{\mathring{w}}}

\newcommand{\mrphi}{\mathring{\phi}}

\newcommand{\mrvgu}{\mathring{\vgu}}

\newcommand{\bp}{ \begin{proof} }
	\newcommand{\ep}{\hfill \end{proof} }
\newcommand{\be}{ \begin{equation} }
\newcommand{\ee}{ \end{equation} }
\newcommand{\tp}{\mathrm{T}}
\newcommand{\dm}{\mathsf{M}} 
\newcommand{\dn}{\mathsf{N}} 
\newcommand{\vgu}{\upsilon} 
\newcommand{\Vgu}{\Upsilon} 

\newcommand{\sr}{\operatorname{sr}}  
\newcommand{\bsr}{\operatorname{bsr}}  
\newcommand{\vmo}{\operatorname{vm}}
\newcommand{\bvmo}{\operatorname{bvm}}

\newcommand{\lp}[1]{l_{#1}(\mathbb{Z})}

\newcommand{\lrs}[3]{(l_{#1}(\mathbb{Z}))^{#2\times #3}}
\newcommand{\sq}{l(\mathbb{Z})} 

\newcommand{\setsp}{\;:\;}     
\newcommand{\td}{\boldsymbol{\delta}}  

\newcommand{\sd}{\mathcal{S}}  
\newcommand{\tz}{\mathcal{T}}  
\newcommand{\cM}{\mathcal{M}}
\newcommand{\cV}{\mathcal{V}}
\newcommand{\ctV}{\tilde{\mathcal{V}}}
\newcommand{\cW}{\mathcal{W}}
\newcommand{\cN}{\mathcal{N}}
\newcommand{\bpo}{\operatorname{bo}} 

\newcommand{\om}[1]{\omega_{#1}}
\newcommand{\ga}[1]{\gamma_{#1}}
\newcommand{\ka}[1]{\mathring{\gamma}_{#1}}

\newcommand{\dC}{\mathbb{C}^d}

\newcommand{\dN}{\mathbb{N}^d}
\newcommand{\dNN}{\mathbb{N}^d_0}

\newcommand{\dR}{\mathbb{R}^d}

\newcommand{\dZ}{\mathbb{Z}^d}

\newcommand{\dLp}[1]{L_{#1}(\mathbb{R}^d)}

\newcommand{\dLrs}[3]{(L_{#1}(\mathbb{R}^d))^{#2\times #3}}

\newcommand{\dlp}[1]{l_{#1}(\mathbb{Z}^d)}

\newcommand{\dlrs}[3]{(l_{#1}(\mathbb{Z}^d))^{#2\times #3}}
\newcommand{\dsq}{l(\mathbb{Z}^d)} 

\begin{document}
	

\title[Balanced multivariate quasi-tight framelets] {Multivariate quasi-tight framelets with high balancing orders derived from any compactly supported refinable vector functions}

\author{Bin Han}
\address[Bin Han]{Department of Mathematical and Statistical Sciences, University of Alberta}
\email{bhan@ualberta.ca}

\author{Ran Lu}
\address[Ran Lu]{Department of Mathematical and Statistical Sciences, University of Alberta}
\email[Corresponding author]{rlu3@ualberta.ca}
	
	\thanks{Research was supported in part by the Natural Sciences and Engineering Research Council of Canada (NSERC)}
	
	
	\makeatletter \@addtoreset{equation}{section} \makeatother
	
	\begin{abstract}
Generalizing wavelets by adding desired redundancy and flexibility,
 framelets (a.k.a. wavelet frames) are of interest and importance in many applications such as image processing and numerical algorithms.
Several key properties of framelets are high vanishing moments for sparse multiscale representation, fast framelet transforms for numerical efficiency, and redundancy for robustness.
However, it is a challenging problem to study and construct multivariate nonseparable framelets, mainly due to their intrinsic connections to factorization and syzygy modules of multivariate polynomial matrices.
Moreover, all known multivariate tight framelets derived from spline refinable scalar functions have only one vanishing moment, and framelets derived from refinable vector functions are barely studied yet in the literature. In this paper, we circumvent the above difficulties through the approach of quasi-tight framelets, which behave almost identically to tight framelets.
Employing the popular oblique extension principle (OEP),
from an arbitrary compactly supported $\dm$-refinable vector function $\phi$ with multiplicity greater than one, we prove that we can always derive
from $\phi$ a compactly supported multivariate quasi-tight framelet such that
\begin{enumerate}
\item[(i)] all the framelet generators have the highest possible order of vanishing moments;
\item[(ii)] its associated fast framelet transform has the highest balancing order and can
be implemented by convolution using only finitely supported filters (i.e., its fast framelet transform is compact).
\end{enumerate}
For a refinable scalar function $\phi$ (i.e., its multiplicity is one), the above item (ii) often cannot be achieved intrinsically but we show that we can always construct a compactly supported OEP-based multivariate quasi-tight framelet derived from $\phi$ satisfying item (i). We point out that constructing OEP-based quasi-tight framelets is closely related to the generalized spectral factorization of Hermitian trigonometric polynomial matrices.
Our proof is critically built on a newly developed result on the normal form of a matrix-valued filter,
which is of interest and importance in itself for greatly facilitating the study of refinable vector functions and multiwavelets/multiframelets.
This paper provides a comprehensive investigation on OEP-based multivariate quasi-tight multiframelets and their associated framelet transforms with high balancing orders. This deepens our theoretical understanding of multivariate quasi-tight multiframelets and their associated fast multiframelet transforms.
\end{abstract}
	
	
	\maketitle
	
	\pagenumbering{arabic}
	
\section{Introduction and Main Results}\label{sec:intro}
	
\subsection{Background}

Framelets (a.k.a. wavelet frames) generalize orthogonal wavelets by adding the desired properties of redundancy in their systems and flexibility in their construction (\cite{chs02,DGM,dh04,dhrs03,han97,rs97}). These extra features greatly improve their performance over orthogonal wavelets in applications such as image denoising and data processing (e.g., see \cite{djls16,ds13,hz14,sel00} and references therein).
The study of multivariate framelets/wavelets is always of interest in both theory and applications. There is a huge amount of literatures on this topic, to mention only a few here, see e.g. \cite{dhrs03,ds13,
eh08,han10,
han14,han15,hl19pp,hm03} and many references therein.
In this paper, we mainly study
compactly supported quasi-tight framelets having several desired properties. Our approach is from the theoretical point of view, to prove the existence and provide a structural characterization of quasi-tight framelets satisfying several desired properties.
To provide the necessary background and
to explain our motivations, let us first recall
some basic concepts and definitions.
Throughout the paper, by $\dm$ we always denote \emph{a dilation matrix}, which is a $d\times d$ integer matrix whose eigenvalues are greater than one in modulus, or equivalently, $\lim_{j\to \infty} \dm^{-j}=0$. Moreover, we define $d_{\dm}:=|\det(\dm)|$.
By $f\in (\dLp{2})^{r\times s}$ we mean that $f$ is an $r\times s$ matrix of square integrable functions in $\dLp{2}$.
In particular, $(\dLp{2})^r:=(\dLp{2})^{r\times 1}$.
Define the inner product by
$\la f,g\ra:=\int_{\dR}f(x)\ol{g(x)}^{\tp}dx$ for $f\in\dLrs{2}{r}{s}$ and $g\in\dLrs{2}{t}{s}$.
Let $\mrphi=(\mrphi_1,\ldots,\mrphi_r)^\tp \in (\dLp{2})^r$ and $\psi=(\psi_1,\ldots,\psi_s)^\tp \in (\dLp{2})^s$. We say that $\{\mrphi;\psi\}$ is \emph{an $\dm$-framelet} in $\dLp{2}$ if
there exist positive constants $C_1$ and $C_2$ such that
\be\label{framelet}
C_1\|f\|_{\dLp{2}}^2\le
\sum_{k\in \dZ} |\la f, \mrphi(\cdot-k)\ra|^2+\sum_{j=0}^\infty \sum_{k\in \dZ}
|\la f, \psi_{\dm^j;k}\ra|^2\le C_2\|f\|_{\dLp{2}}^2, \qquad \forall\, f\in \dLp{2},
\ee
where $|\la f, \psi_{\dm^j;k}\ra|^2:=\|\la f, \psi_{\dm^j;k}\ra\|^2_{l_2}:=\sum_{\ell=1}^r|\la f, \psi_{\ell,\dm^j;k}\ra|^2$, with
$\psi_{\ell,\dm^j;k}:=|\det(\dm)|^{j/2}\psi_\ell(\dm^j\cdot-k)$
and $\psi_{\dm^j;k}:=|\det(\dm)|^{j/2}\psi(\dm^j\cdot-k)=(\psi_{1,\dm^j;k},\ldots,\psi_{s,\dm^j;k})^\tp$. For $\eps_1,\ldots,\eps_s\in \{\pm 1\}$,
we say that $\{\mrphi; \psi\}_{(\eps_1,\ldots,\eps_s)}$ is \emph{a quasi-tight $\dm$-framelet} in $\dLp{2}$ if $\{\mrphi; \psi\}$ is an $\dm$-framelet in $\dLp{2}$ and satisfies
\be\label{qtf:expr}
f=\sum_{\ell=1}^r \sum_{k\in \dZ} \la f, \mrphi_\ell(\cdot-k)\ra \mrphi_\ell(\cdot-k)+
\sum_{j=0}^\infty \sum_{\ell=1}^s \sum_{k\in \dZ}
\eps_\ell \la f, \psi_{\ell,\dm^j;k}\ra \psi_{\ell, \dm^j;k} \qquad \forall\,f\in \dLp{2}
\ee
with the above series converging unconditionally in $\dLp{2}$.
We say that $\{\mrphi;\psi\}$ is \emph{a tight framelet} in $\dLp{2}$ if \er{framelet} holds with $C_1=C_2=1$, or equivalently, $\{\mrphi;\psi\}_{(\eps_1,\ldots,\eps_s)}$ is a quasi-tight $\dm$-framelet with $\eps_1=\cdots=\eps_s=1$.
For a quasi-tight framelet $\{\mrphi;\psi\}_{(\eps_1,\ldots,\eps_s)}$, it is often called a quasi-tight multiframelet (resp. scalar framelet) if the multiplicity of $\mrphi=(\mrphi_1,\ldots,\mrphi_r)^\tp$ is $r>1$ (resp.
$r=1$). For simplicity, we shall use the term framelet to refer both of them unless emphasized.  Moreover, by \cite{han12,han17} on the relations between nonhomogeneous and homogeneous framelets, \eqref{qtf:expr} implies
\be \label{qtf:expr:2}
f=
\sum_{j=-\infty}^\infty \sum_{\ell=1}^s \sum_{k\in \dZ}
\eps_\ell \la f, \psi_{\ell,\dm^j;k}\ra \psi_{\ell, \dm^j;k} \qquad \forall\,f\in \dLp{2}
\ee
with the above series converging unconditionally in $\dLp{2}$.
The multiscale wavelet/framelet representations in \er{qtf:expr} and \eqref{qtf:expr:2} indicate that a quasi-tight framelet behaves almost like a tight framelet.

Most known framelets are constructed from \emph{refinable vector functions} through the \emph{oblique extension principle} (OEP) (see \cite{dhrs03,chs02} and \cite{hanbook}) and such framelets are called \emph{OEP-based framelets}.
By $\dlrs{0}{r}{s}$ we denote the space of all $r\times s$ matrix-valued finitely supported sequences $u=\{u(k)\}_{k\in \dZ}:\dZ\to \C^{r\times s}$ such that $\{k\in \dZ \setsp u(k)\ne 0\}$ is a finite set.
For a vector function $\phi\in (\dLp{2})^r$, we say that $\phi$ is \emph{an $\dm$-refinable vector function} with \emph{a refinement filter/mask} $a\in \dlrs{0}{r}{r}$ if
%
\[
\phi=|\det(\dm)|\sum_{k\in\dZ}a(k)\phi(\dm \cdot-k),\quad
\mbox{or equivalently}, \quad
\wh{\phi}(\dm^{\tp} \xi)=\wh{a}(\xi)\wh{\phi}(\xi),\quad \xi\in \dR,
\]
where $\wh{a}(\xi):=\sum_{k\in \dZ} a(k) e^{-ik\cdot\xi}$ is an $r\times r$ matrix of $2\pi\dZ$-periodic $d$-variate trigonometric polynomials and
$\wh{f}(\xi):=\int_{\dR} f(x) e^{-ix\cdot\xi} dx$ for $\xi\in \dR$ is the Fourier transform of $f\in \dLp{1}$, which can be naturally extended to $\dLp{2}$ functions and tempered distributions.
The integer $r$ is the multiplicity of $\phi$ and $\wh{\phi}$ is the $r\times 1$ vector obtained by taking entry-wise Fourier transform on $\phi$. If $r=1$, then we simply say that $\phi$ is an $\dm$-refinable (scalar) function.

Let $\phi\in (\dLp{2})^r$ be a compactly supported $\dm$-refinable vector function with a refinement filter $a\in\dlrs{0}{r}{r}$.
As a special case of Theorem~\ref{thm:df} in Section~\ref{sec:ffrt}, an OEP-based compactly supported quasi-tight $\dm$-framelet $\{\mrphi;\psi\}_{(\eps_1,\ldots,\eps_s)}$ in $\dLp{2}$ with $\mrphi\in (\dLp{2})^r$ and $\psi\in (\dLp{2})^s$ is derived from $\phi$ through
\be\label{mrphi:psi}
\wh{\mrphi}(\xi):=\wh{\theta}(\xi)\wh{\phi}(\xi),\qquad
\wh{\psi}(\xi):=\wh{b}(\dm^{-\tp}\xi)\wh{\phi}(\dm^{-\tp}\xi)
\ee
for some $\theta\in\dlrs{0}{r}{r}$ and $b\in\dlrs{0}{s}{r}$ such that $\wh{\psi}(0)=0$,
\be \label{Theta}
\ol{\wh{\phi}(0)}^\tp \wh{\Theta}(0)\wh{\phi}(0)=1
\quad \mbox{with}\quad
\wh{\Theta}(\xi):=\ol{\wh{\theta}(\xi)}^{\tp}\wh{\theta}(\xi),
\ee
and $\{a;b\}_{\Theta, (\eps_1,\ldots,\eps_s)}$ is \emph{an OEP-based quasi-tight $\dm$-framelet filter bank}, i.e.,
\be\label{t:fbk}
\ol{\wh{a}(\xi)}^\tp\wh{\Theta}(\dm^{\tp}\xi)\wh{a}(\xi+2\pi \omega)+\ol{\wh{b}(\xi)}^\tp \DG(\eps_1,\ldots,\eps_s)\wh{b}(\xi+2\pi \omega)=\td(\omega)
\wh{\Theta}(\xi),\quad \forall\, \omega\in \Omega_{\dm}
\ee
for all $\xi\in \dR$, where
\be \label{delta:seq}
\td(0):=1 \quad \mbox{and}\quad
\td(x):=0,\qquad \forall\, x\ne 0
\ee
and $\Omega_\dm$ is a particular choice of the representatives of cosets in $[\dm^{-\tp}\dZ]/\dZ$ given by
\be\label{omega:dm}
\Omega_{\dm}:=\{\om{1},\dots,\om{d_{\dm}}\}:=(\dm^{-\tp}\dZ)\cap [0,1)^d\quad \mbox{with}\quad \om{1}:=0.
\ee
In particular, $\{a;b\}_{\Theta}$ is called \emph{an OEP-based  tight $\dm$-framelet filter bank} if $\{a;b\}_{\Theta, (\eps_1,\ldots,\eps_s)}$ is a quasi-tight $\dm$-framelet filter bank and $\eps_1=\dots=\eps_s=1$.

\subsection{Difficulties and motivations on multivariate framelets} \label{subsec:motivation}

We now discuss the difficulties involved in studying and constructing multivariate framelets through OEP with several desired properties.
For any finitely supported matrix-valued filter $b\in\dlrs{0}{s}{r}$, we define
\be \label{Pb}
P_{b;\dm}(\xi):=[\wh{b}(\xi+2\pi\om{1}), \ldots, \wh{b}(\xi+2\pi \om{d_{\dm}})],\qquad\xi\in\dR,
\ee
which is an $s\times (rd_{\dm})$ matrix of $2\pi\dZ$-periodic $d$-variate trigonometric polynomials. One can easily rewrite \er{t:fbk} for a quasi-tight framelet filter bank in the following equivalent matrix form:
\be \label{spectral}
\ol{P_{b;\dm}(\xi)}^\tp \DG(\eps_1,\ldots, \eps_s) P_{b;\dm}(\xi)=
\cM_{a,\Theta}(\xi),
\ee
where
%
\[
\cM_{a,\Theta}(\xi):=\DG\left(\wh{\Theta}(\xi+2\pi\om{1}),\ldots,
\wh{\Theta}(\xi+2\pi \om{d_{\dm}})\right)
-\ol{P_{a;\dm}(\xi)}^\tp \wh{\Theta}(\dm^\tp\xi) P_{a;\dm}(\xi).
\]
The decomposition in \er{spectral} is known as a \emph{generalized matrix spectral factorization} (see \cite{dhacha,dh18pp}) of $\cM_{a,\Theta}$.
For a tight $\dm$-framelet $\{\mrphi;\psi\}$,
since $\eps_1=\cdots=\eps_s=1$,
\eqref{spectral} becomes
the standard spectral factorization problem $\cM_{a,\Theta}(\xi)=\ol{P_{b;\dm}(\xi)}^\tp P_{b;\dm}(\xi)$ (see Subsection~\ref{subsec:df} for details), which requires
\be \label{nonnegative:M}
\cM_{a,\Theta}(\xi)\ge 0, \qquad \forall\, \xi\in \dR.
\ee
If \eqref{nonnegative:M} holds and $d=1$, then the Fej\'er-Riesz lemma guarantees
the existence of a matrix-valued filter
$b\in (\dlp{0})^{s\times r}$ satisfying
$\ol{P_{b;\dm}(\xi)}^\tp P_{b;\dm}(\xi)=
\cM_{a,\Theta}(\xi)$.
But this spectral factorization often fails in dimension $d\ge 2$ for
$\cM_{a,\Theta}(\xi)\ge 0$.
In sharp contrast to one-dimensional framelets (e.g., see \cite{chs02,dhrs03,han15,hanbook,hm05,mo06} and references therein), the lack of the Fej\'er-Riesz lemma for $d\ge 2$ is one of the key difficulties and obstacles for constructing multivariate tight framelets. See \cite{cpss13,cpss15} and Subsection~\ref{subsec:df} for detailed discussion on this issue.
This difficulty motivates us to consider multivariate quasi-tight framelets, which behave almost identically to tight framelets.
The first example of quasi-tight framelets was observed in \cite[Example~3.2.2]{hanbook}. Univariate quasi-tight framelets have been systematically studied in \cite{dh18pp,hl19pp} and multivariate scalar quasi-tight framelets with $\Theta=\td$ and $r=1$ have been investigated in \cite{dhacha}, where $\td\in\dlp{0}$ is the Dirac sequence defined in \eqref{delta:seq}.

The most important feature of the multiscale representations in \eqref{qtf:expr} and \eqref{qtf:expr:2} is their sparsity, which is highly desired for effectively processing multidimensional data.
By $\PL_{m-1}$ we denote the space of all $d$-variate polynomials of degree less than $m$. The sparsity of the multiscale representations in \eqref{qtf:expr} and \eqref{qtf:expr:2} comes from the vanishing moments of $\psi$.
We say that a function $\psi$ has \emph{order $m$ vanishing moments} if
%
\[
\la \pp, \psi\ra=0, \qquad \forall\, \pp\in \PL_{m-1},\quad \mbox{or equivalently},\quad
\wh{\psi}(\xi)=\bo(\|\xi\|^m),\qquad \xi\to 0,
\]
where the notation $f(\xi)=g(\xi)+\bo(\|\xi\|^m)$ as $\xi\to 0$ simply means $\partial^\mu f(0)=\partial^\mu g(0)$ for all $\mu=(\mu_1,\ldots,\mu_d)^\tp\in \dNN$ with $|\mu|:=\mu_1+\cdots+\mu_d<m$.
We define $\vmo(\psi):=m$ with $m$ being the largest such integer.
It is easy to deduce from \er{qtf:expr} that a necessary condition for all framelet generators $\psi_\ell, \ell=1,\ldots,s$ to have order $m$ vanishing moments is the following polynomial preservation property:
\be \label{qi:order}
\sum_{k\in \dZ} \la \pp, \mrphi(\cdot-k)\ra \mrphi(\cdot-k):=\sum_{\ell=1}^r \sum_{k\in \dZ} \la \pp, \mrphi_\ell(\cdot-k)\ra \mrphi_\ell(\cdot-k)=\pp,\qquad \forall\, \pp\in \PL_{m-1},
\ee
which plays a crucial role in approximation theory and numerical analysis for the convergence rate of the associated approximation/numerical scheme.
Using the Fourier transform and $\wh{\mrphi}(\xi)=\wh{\theta}(\xi)\wh{\phi}(\xi)$ in \eqref{mrphi:psi}, it is well known in the approximation theory (e.g., see \cite[Proposition~5.5.2]{hanbook})
that \eqref{qi:order} is equivalent to $\ol{\wh{\phi}(\xi)}^\tp \wh{\Theta}(\xi) \wh{\phi}(\xi)=1+\bo(\|\xi\|^m)$ as $\xi \to 0$ and
\be \label{vm:moments}
\ol{\wh{\phi}(\xi)}^\tp \wh{\Theta}(\xi) \wh{\phi}(\xi+2\pi k)=\bo(\|\xi\|^m), \qquad  \xi\to 0, \quad \forall\; k\in \dZ\bs\{0\},
\ee
where $\wh{\Theta}$ is defined in \eqref{Theta}.
Multiplying $\ol{\wh{\phi}(\xi)}^\tp$ on the left side and $\wh{\phi}(\xi)$ on the right side of \eqref{t:fbk}, we can further deduce from
\eqref{t:fbk} with $\omega=0$ and \eqref{Theta}
that $\psi$ has order $m$ vanishing moments implies
\be \label{vm:0}
1-\ol{\wh{\phi}(\xi)}^\tp \wh{\Theta}(\xi) \wh{\phi}(\xi)=\bo(\|\xi\|^{2m}),\quad \xi\to 0.
\ee
Moreover, one can conclude
that the framelet generator $\psi$
in a tight framelet $\{\mrphi;\psi\}$
has order $m$ vanishing moments if and only if \eqref{vm:moments} and \eqref{vm:0} hold.
The main goal of the OEP is to improve vanishing moments of framelets $\psi$ by properly constructing $\theta\in (\dlp{0})^{r\times r}$ such that \er{vm:moments}, \eqref{vm:0} and \eqref{spectral} are satisfied with $\wh{\Theta}$ being defined in \eqref{Theta} (see \cite{chs02,dh04,dhrs03,han09,han12,han14,hanbook,hm03} for detailed discussion).
For tight framelets, \eqref{nonnegative:M} must also hold.
For dimension one, the existence of a filter $\theta \in (\lp{0})^{r\times r}$ satisfying \eqref{nonnegative:M},
\eqref{vm:moments} and  \eqref{vm:0}
has been established in \cite{hm05} for $r=1$ and in \cite{mo06} for $r>1$.
However, the existence for a desired filter $\theta\in(\dlp{0})^{r\times r}$ remains unresolved for $d\ge 2$.

Due to the above difficulties, most papers in the literature (e.g., see \cite{cpss13,cpss15,cj00,dhacha,han14,hjsz18,js15,rs98} and references therein) study multivariate framelets only for the particular case $r=1,\Theta=\td$ and special choices of $\phi$, i.e., $\mrphi=\phi$ and $\wh{\Theta}=1$.
Indeed, many known refinable scalar functions such as spline refinable functions satisfy \eqref{vm:moments} with $\Theta=\td$ for a large positive integer $m$, which guarantees that the integer shifts of $\phi$ provide $m$ approximation order for approximating functions.
But \eqref{vm:0} with $\Theta=\td$  can only hold with $m=1$ for most known refinable (scalar) functions including all spline functions.
Hence, it is not surprising that most known multivariate tight framelets including those derived from all spline functions can have only one vanishing moment.
Using $\Theta=\td$ loses the main advantage of OEP for improving vanishing moment orders of the framelet $\psi$.

Suppose now that we could construct a desired filter $\theta\in (\dlp{0})^{r\times r}$ satisfying \eqref{vm:moments} and \eqref{vm:0} (as well as \eqref{nonnegative:M})
such that a compactly supported quasi-tight framelet $\{\mrphi; \psi\}_{(\eps_1,\ldots,\eps_1)}$ can be derived from a refinable vector function $\phi$.
The multiscale representations in \eqref{qtf:expr} and \eqref{qtf:expr:2} using the compactly supported quasi-tight framelet $\{\mrphi; \psi\}_{(\eps_1,\ldots,\eps_1)}$
appear to be perfect,
but two serious difficulties still remain if $\Theta\ne \td I_r$.
Here we only briefly address these two issues, see Section~\ref{sec:ffrt} for detailed discussion.
Using the definition of $\mrphi$ and $\psi$ in \eqref{mrphi:psi}, we observe from $\wh{\phi}(\dm^\tp \xi)=\wh{a}(\xi)\wh{\phi}(\xi)$ that
\be \label{mphi:mpsi}		 \wh{\mrphi}(\dm^{\tp}\xi)=\wh{\mra}(\xi)\wh{\mrphi}(\xi),\qquad \wh{\psi}(\dm^{\tp}\xi)=\wh{\mrb}(\xi)\wh{\mrphi}(\xi),\qquad\xi\in\R^d,
\ee
where
\be \label{mab}		 \wh{\mra}(\xi):=\wh{\theta}(\dm^{\tp} \xi) \wh{a}(\xi) \wh{\theta}(\xi)^{-1}
\qquad \mbox{and}\qquad		 \wh{\mrb}(\xi):=\wh{b}(\xi) \wh{\theta}(\xi)^{-1}.
\ee
As we shall discuss in Section~\ref{sec:ffrt},
the underlying discrete multiframelet transform employing
a quasi-tight framelet filter bank $\{a;b\}_{\Theta, (\eps_1,\ldots,\eps_1)}$ actually employs the filters $\mra$ and $\mrb$, which are often not finitely supported (i.e., $\wh{\mra}$ and $\wh{\mrb}$ are not matrices of $2\pi\dZ$-periodic trigonometric polynomials) if $\Theta\ne \td I_r$. Thus, deconvolution is unavoidable and this greatly hinders the efficiency of their associated framelet transform.
There is an additional difficulty for multiframelets with $r>1$:
the vanishing moments of $\psi$ do not necessarily transfer into sparsity of its associated discrete multiframelet transform. This issue is known as the balancing property in the literature for multiwavelets and multiframelets (\cite{cj00,han09,han10,hanbook,lv98,sel00}), where most known constructed multiframelets often have a much lower balancing order than its order of vanishing moments. The balancing property of a multiframelet is essential for the sparsity of the associated multiframelet transform, see Section~\ref{sec:ffrt} for details.

\subsection{Main results and contributions}

In this paper, we shall resolve all the above difficulties and issues in Subsection~\ref{subsec:motivation}
by taking the approach of OEP and quasi-tight framelets.
The equation \eqref{vm:moments} for the approximation property of a refinable vector function $\phi$ with a refinement filter $a$ is intrinsically linked to the sum rules of the refinement filter $a$ (e.g., see \cite[Proposition~5.5.2 and Theorem~5.5.4]{hanbook} and \cite{han03,jj02}). We say that a filter $a\in (\dlp{0})^{r\times r}$ has \emph{order $m$ sum rules with respect to $\dm$} with a matching filter $\vgu\in \dlrs{0}{1}{r}$ if $\wh{\vgu}(0)\ne 0$ and
\be \label{sr}
\wh{\vgu}(\dm^{\tp}\xi)\wh{a}(\xi+2\pi\omega)=\td(\omega)\wh{\vgu}(\xi)+
\bo(\|\xi\|^m),\quad \xi\to 0,\quad\forall\, \omega\in\Omega_{\dm}.
\ee
In particular, we define $\sr(a,\dm):=m$ with $m$ being the largest possible integer in \eqref{sr}. For a quasi-tight $\dm$-framelet $\{\mrphi;\psi\}_{(\eps_1,\ldots,\eps_s)}$, in fact, $\vmo(\psi)\le \sr(a,\dm)$ no matter how we choose the filter $\theta$.

We say that a fast transform is \emph{compact} if it can be implemented by convolution using only finitely supported filters.
For a filter $\theta\in (\dlp{0})^{r\times r}$, we say that $\wh{\theta}$ (or simply $\theta$) is \emph{strongly invertible} if $\wh{\theta}^{-1}$ is also a matrix of $2\pi\dZ$-periodic trigonometric polynomials.
Our main result is as follows:

\begin{theorem}\label{thm:qtf}
	Let $\dm$ be a $d\times d$ dilation matrix and $\phi\in (\dLp{2})^r$ be a
compactly supported $\dm$-refinable vector function satisfying $\wh{\phi}(\dm^{\tp} \xi)=\wh{a}(\xi)\wh{\phi}(\xi)$ with $\wh{\phi}(0)\ne 0$ and
a matrix-valued filter $a\in \dlrs{0}{r}{r}$.
	Suppose that the filter $a$ has order $m$ sum rules with respect to $\dm$ satisfying \eqref{sr} with a matching filter $\vgu\in \dlrs{0}{1}{r}$ such that $\wh{\vgu}(0)\wh{\phi}(0)=1$.
Let $\dn$ be a $d\times d$ integer matrix with $|\det(\dn)|=r$.
If $r\ge 2$, then there exist filters $b\in \dlrs{0}{s}{r}$, $\theta\in \dlrs{0}{r}{r}$ and $\eps_1,\ldots,\eps_s\in \{\pm 1\}$ such that
	\begin{enumerate}
		\item[(1)] $\{\mrphi; \psi\}_{(\eps_1,\ldots,\eps_s)}$ is a compactly supported quasi-tight $\dm$-framelet in $\dLp{2}$ such that $\psi$ has order $m$ vanishing moments,
where $\mrphi$ and $\psi$ are defined in \eqref{mrphi:psi}.
Moreover,
$\mrphi$ and $\psi$ satisfy the refinable structure in
\eqref{mphi:mpsi} with the filters $\mra,\mrb$ being defined in \eqref{mab}.

		\item[(2)] $\wh{\theta}$ is \emph{strongly invertible}, i.e., $\wh{\theta}^{-1}$ is also a matrix of $2\pi\dZ$-periodic trigonometric polynomials.
		
		\item[(3)]
$\{a;b\}_{\Theta, (\eps_1,\ldots,\eps_s)}$ and
$\{\mra; \mrb\}_{\td I_r, (\eps_1,\ldots,\eps_s)}$ are finitely supported quasi-tight $\dm$-framelet filter banks.
		
		\item[(4)] 	The associated discrete multiframelet transform employing $\{\mra; \mrb\}_{\td I_r, (\eps_1,\ldots,\eps_s)}$ is compact and order $m$ $E_{\dn}$-balanced,
where $E_{\dn}$ is the vector conversion operator in \er{vec:con} (see Section~\ref{sec:ffrt} for details).
\end{enumerate}
	\end{theorem}

For $r=1$, a similar result is given in Corollary~\ref{cor:qtf:r:1} satisfying only item (1); but we have to unavoidably give up all the desired properties in items (2)--(4).
Our proof of Theorem~\ref{thm:qtf} is built on the following result on the normal form of a matrix-valued filter, which is of interest and importance in itself for greatly facilitating the study of refinable vector functions, multiwavelets and multiframelets.

\begin{theorem}\label{thm:normalform}
Let $\dm$ be a $d\times d$ dilation matrix and $\phi$ be a vector of
compactly supported distributions satisfying $\wh{\phi}(\dm^{\tp} \xi)=\wh{a}(\xi)\wh{\phi}(\xi)$ with $\wh{\phi}(0)\ne 0$ and
a finitely supported matrix-valued filter $a\in \dlrs{0}{r}{r}$.
Suppose that the filter $a$ has order $m$ sum rules with respect to $\dm$ satisfying \eqref{sr} with a matching filter $\vgu\in \dlrs{0}{1}{r}$ such that $\wh{\vgu}(0)\wh{\phi}(0)=1$.
If $r\ge 2$, then for any positive integer $n\in \N$, there exists a strongly invertible $r\times r$ matrix $\wh{U}$ of $2\pi\dZ$-periodic trigonometric polynomials such that the following properties hold:
	
\begin{enumerate}
\item[(1)] Define
$\wh{\mathring{\vgu}}(\xi):=(\wh{\mathring{\vgu}_1}(\xi),\ldots,	 \wh{\mathring{\vgu}_r}(\xi)):=\wh{\vgu}(\xi)\wh{U}(\xi)^{-1}$ and
$\wh{\mathring{\phi}}(\xi):=(\wh{\mathring{\phi}_1}(\xi),\ldots,		 \wh{\mathring{\phi}_r}(\xi))^\tp:=\wh{U}(\xi)\wh{\phi}(\xi)$.
Then
\be \label{normalform:phi}		 \wh{\mathring{\phi}_1}(\xi)=1+\bo(\|\xi\|^n)
		\quad \mbox{and}\quad \wh{\mathring{\phi}_\ell}(\xi)=\bo(\|\xi\|^n),\quad \xi\to 0,\quad \ell=2,\ldots,r,
\ee		 \be\label{normalform:vgu}\wh{\mathring{\vgu}_1}(\xi)=1+\bo(\|\xi\|^m)
		\quad \mbox{and}\quad
		 \wh{\mathring{\vgu}_\ell}(\xi)=\bo(\|\xi\|^m),\quad \xi\to 0, \quad\ell=2,\ldots,r.\ee	
		
		\item[(2)] Define a finitely supported matrix-valued filter $\mra\in\dlrs{0}{r}{r}$ by $\wh{\mathring{a}}(\xi):=\wh{U}(\dm^{\tp}\xi) \wh{a}(\xi)\wh{U}(\xi)^{-1}$. Then the filter $\mra$ takes the \emph{standard $(m,n)$-normal form}, i.e.,
\be \label{normalform}
		\wh{\mra}(\xi)=\left[
		 \begin{matrix}\wh{\mra_{1,1}}(\xi) & \wh{\mra_{1,2}}(\xi)\\
			\wh{\mra_{2,1}}(\xi) & \wh{\mra_{2,2}}(\xi)\end{matrix}\right],
		\ee
		where $\wh{\mra_{1,1}},\wh{\mra_{1,2}}, \wh{\mra_{2,1}}$ and $\wh{\mra_{2,2}}$ are $1\times 1$, $1\times (r-1)$, $(r-1)\times 1$ and $(r-1)\times (r-1)$ matrices of $2\pi\dZ$-periodic trigonometric polynomials such that
\begin{align}	 &\wh{\mra_{1,1}}(\xi)=1+\bo(\|\xi\|^n),\quad \wh{\mra_{1,1}}(\xi+2\pi\omega)=\bo(\|\xi\|^m),\quad \xi\to 0,\quad\forall \omega\in\Omega_{\dm}\setminus\{0\},\label{mra:11}\\
&\wh{\mra_{1,2}}(\xi+2\pi\omega)=\bo(\|\xi\|^m),\quad \xi\to 0,\quad\forall \omega\in\Omega_{\dm},\label{mra:12}\\ &\wh{\mra_{2,1}}(\xi)=\bo(\|\xi\|^n),\quad \xi\to 0.\label{mra:21}
\end{align}
Moreover, $\wh{\mrphi}(\dm^{\tp}\xi)=\wh{\mra}(\xi)\wh{\mrphi}(\xi)$ and the new filter $\mra$ has order $m$ sum rules with respect to $\dm$ with the matching filter $\mathring{\vgu}\in (\dlp{0})^{1\times r}$ in item (1).
		
\item[(3)] Define $\|\wh{\phi}(\xi)\|^2:=\|\wh{\phi}(\xi)\|^2_{l_2}:=
    \ol{\wh{\phi}(\xi)}^\tp \wh{\phi}(\xi)$.
If in addition
\be\label{moment:special}
\wh{\vgu}(\xi)=
\|\wh{\phi}(\xi)\|^{-2}
\ol{\wh{\phi}(\xi)}^{\tp}
+\bo(\|\xi\|^m),\quad\xi\to 0,
\ee
then the strongly invertible $\wh{U}$ can satisfy the following ``almost orthogonality" condition:
		 \be\label{eq:ortho}\ol{\wh{U}(\xi)}^{-\tp}\wh{U}(\xi)^{-1}=\DG\left(\|\wh{\phi}(\xi)\|^2,\|\wh{u_2}(\xi)\|^2,\dots, \|\wh{u_r}(\xi)\|^2\right)+\bo(\|\xi\|^{\tilde{n}}),\quad\xi\to 0,\ee
where $\wh{u_j}$ is the $j$-th column of the matrix $\wh{U}^{-1}$ for $j=2,\dots,r$ and $\tilde{n}:=\max(m,n)$.
\end{enumerate}
Conversely, if there exists a strongly invertible matrix $\wh{U}$ of $2\pi\dZ$-periodic trigonometric polynomials such that items (1) and (2) and \er{eq:ortho} hold with $n\ge m$, then \er{moment:special} must hold.
\end{theorem}

Here are some remarks and contributions of our main results in Theorems~\ref{thm:qtf} and~\ref{thm:normalform}.
\begin{enumerate}
	
\item To prove the main result Theorem~\ref{thm:qtf}, the problem essentially boils down to searching for a suitable filter $\theta\in\dlrs{0}{r}{r}$ and obtaining a generalized spectral factorization of $\cM_{a,\Theta}$ as in \er{spectral} such that all desired properties in Theorem~\ref{thm:qtf} are satisfied. The proof of the result greatly demonstrates the flexibility of the generalized matrix spectral factorization over the regular matrix spectral factorization.
	
\item The one-dimensional version of Theorem~\ref{thm:qtf} was established in \cite{hl19pp}. However, due to the aforementioned difficulties on multivariate multiframelets,
    Theorem \ref{thm:qtf} is not a simple generalization of \cite{hl19pp}. Several new challenges and difficulties are involved in the study of multivariate multiframelets. For example, the factorization technique in the case $d=1$ for solving \eqref{spectral}
    no longer works when $d\ge 2$. It is well known that for $d=1$, a $2\pi$-periodic trigonometric polynomial has order $m$ vanishing moments if and only if it is divisible by $(1-e^{-i\xi})^m$. This fact is the key in the study of univariate framelets with high vanishing moments in \cite{han15,hanbook,hl19pp}. However, there is no such corresponding factor playing the role of $(1-e^{-i\xi})^m$ when $d\ge 2$.

	\item The key ingredient to prove Theorem~\ref{thm:qtf} is a newly developed normal form of a multivariate matrix-valued filter in Theorem~\ref{thm:normalform}, which greatly facilitates the study of refinable vector functions and multiwavelets/multiframelets. Our result on the filter normal form significantly improves \cite[Proposition~2.4]{han03} and \cite[Theorem~5.1]{han10}, and plays a key role in the study of multivariate quasi-tight multiframelets with high vanishing moments and the balancing property.
		
		\item The balancing property is more complicated in high dimensions than in dimension one. To be brief, the balancing property describes how efficient a discrete multiframelet transform is when handling vectorized input data. Here by vectorized input data we mean that the input data is obtained by converting a scalar data into a vector using a standard vector conversion process. When the multiplicity is given, there is only one standard way to convert a scalar data into a vector data when $d=1$. However, there are multiple ways to perform the vectorization with a given multiplicity when $d\ge 2$. We will discuss the balancing property of a discrete multiframelet transform in detail in Section~\ref{sec:ffrt}.
		
		\item Theorem~\ref{thm:qtf} demonstrates great advantages of quasi-tight framelets. A multivariate quasi-tight framelet with the highest possible vanishing moments can be always obtained from any arbitrarily given compactly supported refinable vector function.
This is not the case for constructing multivariate tight framelets. To our best knowledge, all existing constructions of multivariate nonseparable tight framelets are designed for special low-pass filters and
often have low vanishing moments
(see e.g. \cite{Ehler07,eh08,
hjsz18,hz14,kps16,lj06,rs98,sz16}).
\end{enumerate}

\subsection{Paper structure}

The structure of the paper is as follows.
In Section~\ref{sec:ffrt},
we shall provide further explanations on the difficulty on multivariate framelets with high vanishing moments and their associated discrete multiframelet transforms.
We first briefly review dual framelets in Section~\ref{sec:ffrt}. Next,
we shall investigate discrete multiframelet transforms employing multivariate dual framelet filter banks. Then we will discuss the perfect reconstruction and the balancing property of discrete multiframelet transforms at the end of Section~\ref{sec:ffrt}.
In Section~\ref{sec:normalform}, we shall prove Theorem~\ref{thm:normalform} on a normal form of a matrix-valued filter, with some demonstrations and explanation of the importance of the normal form theory in the study of multivariate refinable vector functions and multiwavelets/multiframelets. In Section~\ref{sec:qtf}, we shall prove Theorem~\ref{thm:qtf} on the existence of quasi-tight multiframelets satisfying all desired properties. We will further study the structure of balanced multivariate quasi-tight multiframelets derived through OEP, and provide an algorithm for construction.

\section{Discrete Multiframelet Transforms and Balancing Property}
\label{sec:ffrt}

In this section, to further explain our motivations, we shall investigate various properties such as compactness and balancing property of a discrete multiframelet transform employing an OEP-based dual framelet filter bank.

\subsection{Dual framelet filter banks}
\label{subsec:df}
Since we shall study a discrete multiframelet transform employing dual framelets,
let us briefly recall the definition of multivariate dual framelets, which help us better understand our motivations and provide some preliminaries for the proofs of our main results.

Let $\mrphi,\tilde{\mrphi}\in (\dLp{2})^r$ and $\psi,\tilde{\psi}\in (\dLp{2})^s$.
We say that $( \{\mrphi; \psi\},\{\tilde{\mrphi}; \tilde{\psi}\})$ is \emph{a dual $\dm$-framelet} in $\dLp{2}$ if both $\{\mrphi; \psi\}$ and $\{\tilde{\mrphi}; \tilde{\psi}\}$ are $\dm$-framelets in $\dLp{2}$ (i.e., they satisfy \eqref{framelet}) and
\be\label{framelet:expr}
\la f,g\ra=\sum_{k\in \dZ} \la f,\mrphi(\cdot-k) \ra \la \tilde{\mrphi}(\cdot-k),g\ra+\sum_{j=0}^\infty \sum_{k\in \dZ}
\la f, \psi_{\dm^j;k}\ra \la \tilde{\psi}_{\dm^j;k},g\ra,\quad \forall f,g\in \dLp{2},
\ee
with the above series converging absolutely. As a direct consequence of the identity in \eqref{framelet:expr} (e.g., see \cite[Proposition~5]{han12} and \cite[Theorem~4.3.5]{hanbook}),  we have the following wavelet/framelet representations: for $f\in \dLp{2}$,
\[
f=\sum_{k\in \dZ} \la f,\mrphi(\cdot-k) \ra \tilde{\mrphi}(\cdot-k)+\sum_{j=0}^\infty \sum_{k\in \dZ}
\la f,\psi_{\dm^j;k}\ra  \tilde{\psi}_{\dm^j;k}
=
\sum_{j\in \Z} \sum_{k\in \dZ}
\la f, \psi_{\dm^j;k}\ra \tilde{\psi}_{\dm^j;k}
\]
with the above series converging unconditionally in $\dLp{2}$.
Note that $\{\mrphi;\psi\}$ is a tight $\dm$-framelet in $\dLp{2}$ if and only if $(\{\mrphi;\psi\},\{\mrphi;\psi\})$ is a dual $\dm$-framelet in $\dLp{2}$. It is also obvious that $\{\mrphi; \psi\}_{(\eps_1,\ldots,\eps_s)}$ is a quasi-tight $\dm$-framelet if and only if $(\{\mrphi;\psi\}, \{\mrphi;\tilde{\psi}\})$ is a dual $\dm$-framelet with $\tilde{\psi}:=(\eps_1 \psi_1,\ldots,\eps_s \psi_s)^\tp$ which is almost $\psi$ except possible sign change.

To construct framelets from refinable vector functions, an \emph{oblique extension principle (OEP)} was introduced. The univariate scalar framelet version of the OEP was introduced in \cite{dhrs03} and independently in \cite{chs02}. The OEP for multivariate scalar framelets was studied in \cite[Proposition~3.3]{han03-0} and \cite{han12,han14}.
The univariate multiframelet version of the OEP was studied in \cite[Theorem~3.1]{hm03} and \cite[Theorem~1.1]{han09} (also see
\cite[Theorem~6.4.1]{hanbook}).
Its corresponding multivariate version is as follows:

\begin{theorem}[\emph{Oblique Extension Principle (OEP)}]\label{thm:df}
	Let $\dm$ be a $d\times d$ dilation matrix. Let $\theta,\tilde{\theta},a,\tilde{a}\in \dlrs{0}{r}{r}$ and $\phi,\tilde{\phi}\in (\dLp{2})^r$ be compactly supported $\dm$-refinable vector functions with refinement filters $a$ and $\tilde{a}$, respectively.
For matrix-valued filters $b,\tilde{b}\in \dlrs{0}{s}{r}$,
	define $\mrphi,\psi$ as in \eqref{mrphi:psi} and
%
%
\[
\wh{\tilde{\mrphi}}(\xi):=\wh{\tilde{\theta}}(\xi)\wh{\tilde{\phi}}(\xi),\quad \wh{\tilde{\psi}}(\xi):=\wh{\tilde{b}}(\dm^{-\tp}\xi)\wh{\tilde{\phi}}(\dm^{-\tp}\xi).
\]
Then $(\{\mrphi;\psi\}, \{\tilde{\mrphi}; \tilde{\psi}\})$ is a dual $\dm$-framelet in $\dLp{2}$ if the following conditions are satisfied:
	\begin{enumerate}
		\item[(1)] $\ol{\wh{\phi}(0)}^\tp \wh{\Theta}(0)\wh{\tilde{\phi}}(0)=1$ with $\wh{\Theta}(\xi):=\ol{\wh{\theta}(\xi)}^\tp \wh{\tilde{\theta}}(\xi)$;

\item[(2)] all entries in $\psi$ and $\tilde{\psi}$ have at least one vanishing moment, i.e., $\wh{\psi}(0)=\wh{\tilde{\psi}}(0)=0$.

		\item[(3)] $(\{a;b\},\{\tilde{a};\tilde{b}\})_{\Theta}$ forms an OEP-based dual $\dm$-framelet filter bank, i.e.,
		\be \label{dffb}
		\ol{\wh{{a}}(\xi)}^\tp \wh{\Theta}(\dm^{\tp} \xi){\wh{\tilde{a}}(\xi+2\pi \omega)}+\ol{\wh{{b}}(\xi)}^\tp {\wh{\tilde{b}}(\xi+2\pi \omega)}=\td(\omega)\wh{\Theta}(\xi) ,\quad\forall \xi\in\dR,\quad\omega\in\Omega_{\dm},\ee
		where $\td$ and $\Omega_{\dm}$ are defined as in \er{delta:seq} and \er{omega:dm}, respectively.
	\end{enumerate}
\end{theorem}

By Theorem~\ref{thm:df}, the key step to obtain an OEP-based dual framelet is the construction of an OEP-based dual framelet filter bank $(\{a;b\},\{\tilde{a};\tilde{b}\})_{\Theta}$ satisfying \eqref{dffb} of Theorem~\ref{thm:df}.
Let us now rewrite \eqref{dffb} into a matrix form below.
For $\gamma\in \dZ$ and $u\in (\dsq)^{s\times r}$, the \emph{$\gamma$-coset sequence} of $u$ with respect to $\dm$ is the sequence $u^{[\gamma;\dm]}\in (\dsq)^{s\times r}$ given by
$$
u^{[\gamma;\dm]}(k)=u(\gamma+\dm k),\quad k\in\dZ.
$$
Trivially,
$\wh{u}(\xi)=\sum_{\gamma\in\Gamma_{\dm}}
\wh{u^{[\gamma;\dm]}}(\dm^{\tp}\xi)e^{-i\gamma\cdot\xi}$,
where $\Gamma_{\dm}$ is a complete set of canonical representatives of the quotient group $\dZ/[\dm\dZ]$ given by
\be\label{enum:gamma}
\Gamma_{\dm}:=\{\ga{1},\dots,\ga{d_{\dm}}\}
\quad \mbox{with}\quad \ga{1}:=0
\quad \mbox{such that $\dZ$ is the disjoint union of}\; \Gamma_{\dm}+\dm\dZ.
\ee
Let $\Omega_{\dm}$ be defined in
\eqref{omega:dm} and note that $\#\Omega_{\dm}=\#\Gamma_{\dm}=d_\dm:=|\det(\dm)|$.
Define $\FF_{r;\dm}(\xi)$ to be the $(rd_{\dm})\times (rd_{\dm})$ matrix below
\be\label{Fourier}
\FF_{r;\dm}(\xi):=\left(e^{-i\ga{l}\cdot(\xi+2\pi\om{k})}I_r\right)_{1\leq l,k\leq d_\dm}.
\ee	

For $\omega\in\Omega_{\dm}$ and $u\in\dlrs{0}{r}{r}$, let $D_{u,\omega;\dm}(\xi)$ and $E_{u,\omega;\dm}(\xi)$ be the $(rd_{\dm} )\times (rd_{\dm})$ block matrices, whose $(l,k)$-th $r\times r$ blocks are given by
\be\label{Duni}(D_{u,\omega;\dm}(\xi))_{l,k}:=
\begin{cases}\wh{u}(\xi+2\pi\omega), &\text{if }\om{l}+\omega-\om{k}\in\dZ.\\
	0, &\text{otherwise},\end{cases}
\ee
and
\be\label{Euni}(E_{u,\omega;\dm}(\xi))_{l,k}:=
\wh{u^{[\ga{k}-\ga{l};\dm]}}
(\xi)e^{-i\ga{k}\cdot(2\pi\omega)}.
\ee
Following the lines of the proof of \cite[Lemma 7]{dhacha},  we have
\be\label{DEF}
\FF_{r;\dm}(\xi)D_{u,\omega;\dm}(\xi)
\ol{\FF_{r;\dm}(\xi)}^{\tp}=d_{\dm} E_{u,\omega;\dm}(\dm^{\tp}\xi),\qquad \xi\in\dR, \quad \omega\in \Omega_{\dm}.
\ee	
Recall that $P_{u;\dm}(\xi):=\big[\wh{u}(\xi+2\pi\om{1}),
\wh{u}(\xi+2\pi\om{2}),\dots,\wh{u}(\xi+2\pi\om{d_{\dm}})\big]$
in \eqref{Pb}.
It is straightforward to check that
$P_{u;\dm}(\xi)=Q_{u;\dm}(\dm^\tp \xi) \FF_{r;\dm}(\xi)$, where
\be \label{Qu}
Q_{u;\dm}(\xi):=\big[\wh{u^{[\ga{1};\dm]}}(\xi),\wh{u^{[\ga{2};\dm]}}(\xi),\dots,\wh{u^{[\ga{d_{\dm}};\dm]}}(\xi)\big].
\ee
Since $\ol{\FF_{r;\dm}(\xi)}^{\tp}\FF_{r;\dm}(\xi)=d_{\dm} I_{d_\dm r}$,
it is trivial to observe that
$P_{u;\dm}(\xi)\ol{\FF_{r;\dm}(\xi)}^{\tp}=
d_{\dm} Q_{u;\dm}(\dm^\tp \xi)$.
Now by \eqref{DEF}, it is clear that \er{dffb} is equivalent to
\be\label{dffb:coset}
\cN_{a,\tilde{a},\Theta}(\xi)=\ol{Q_{b;\dm}(\xi)}^\tp Q_{\tilde{b};\dm}(\xi)
\quad \mbox{with}\quad
\cN_{a,\tilde{a},\Theta}(\xi):=
d_{\dm}^{-1}E_{\Theta,0;\dm}(\xi)-
\ol{Q_{a;\dm}(\xi)}^\tp \wh{\Theta}(\xi) Q_{\tilde{a};\dm}(\xi).
\ee
To obtain a dual framelet filter bank, we have to factorize the matrix $\cN_{a,\tilde{a},\Theta}$ of $2\pi\dZ$-periodic trigonometric polynomials such that the constructed filters $b$ and $\tilde{b}$ have desired orders of vanishing moments. Such a factorization for \eqref{dffb:coset} with $d\ge 2$
is linked to syzygy modules of multivariate polynomial matrices (see \cite{Ehler07,eh08,han03-0})
and is difficult to solve, because there does not exist special factors, played by $(1-e^{-i\xi})^m$ in dimension one, for the vanishing moments of filters $b$ and $\tilde{b}$.
Note that $\{a;b\}_{\Theta,(\eps_1,\ldots,\eps_s)}$ is a quasi-tight $\dm$-framelet filter bank if and only if $(\{a;b\},\{a;\tilde{b}\})_\Theta$ is a dual $\dm$-framelet filter bank with $\tilde{b}:=\DG(\eps_1,\ldots,\eps_s) b$.
By $\cM_{a,\Theta}(\xi)=\ol{\FF_{r;\dm}(\xi)}^\tp \cN_{a,a,\Theta}(\dm^\tp \xi) \FF_{r;\dm}(\xi)$,
the matrix factorization in \eqref{spectral} is equivalent to the following special form of \eqref{dffb:coset}:
\be \label{spectral:1}
\ol{Q_{b;\dm}(\xi)}^\tp \DG(\eps_1,\ldots,\eps_s) Q_{b;\dm}(\xi)=\cN_{a,a,\Theta}(\xi).
\ee
In particular, for a tight $\dm$-framelet filter bank $\{a;b\}_{\Theta}$, the above identity in \eqref{spectral:1} with $\eps_1=\cdots=\eps_s=1$ is simply the standard matrix spectral factorization problem.
Therefore, to construct a multivariate tight framelet through OEP, we must obtain a spectral factorization of the multivariate Hermitian matrix $\cN_{a,a,\Theta}$ of $2\pi\dZ$-periodic trigonometric polynomials, requiring $\cN_{a,a,\Theta}(\xi)\ge 0$ for all $\xi\in \dR$.
As well explained in \cite{cpss13,cpss15}, obtaining a spectral factorization of a Hermitian trigonometric polynomial matrix is extremely difficult when $d\ge 2$, and is known to be a challenging problem in algebraic geometry.

\subsection{Discrete multiframelet transforms}

Following \cite{han13}, we now study a discrete multiframelet transform employing an OEP-based dual framelet filter bank.
By $(\dsq)^{s\times r}$ we denote the linear space of all sequences $v: \dZ \rightarrow \C^{s\times r}$. For a filter $a\in \dlrs{0}{r}{r}$, we define the filter $a^\star$ via $\wh{a^\star}(\xi):=\ol{\wh{a}(\xi)}^\tp$, or equivalently,
$a^\star(k):=\ol{a(-k)}^\tp$ for all $k\in \dZ$.
We define the \emph{convolution} of two filters via
\[
[v*a](n):=\sum_{k\in \Z} v(k) a(n-k),\quad n\in \dZ,\quad v\in(\dsq)^{s\times r},\quad a\in\dlrs{0}{r}{r}.
\]
Let $\dm$ be a $d\times d$ dilation matrix and $a\in\dlrs{0}{r}{r}$ be a filter. We define the \emph{subdivision} and \emph{transition} operators as follows:
\[
\sd_{a,\dm} v=|\det(\dm)|^{\frac{1}{2}}[v\uparrow\dm]*a,\qquad\tz_{a,\dm} v=|\det(\dm)|^{\frac{1}{2}}[v*a^\star](\dm \cdot),\quad v\in(\dsq)^{s\times r},
\]
where $v\uparrow\dm$ is the sequence defined by $[v\uparrow\dm](k)=v(\dm^{-1}k)$ if $k\in \dm\dZ$ and $[v\uparrow\dm](k)=0$ otherwise. Let $a,\tilde a, \theta, \tilde{\theta}\in \lrs{0}{r}{r}$ and
$b,\tilde{b}\in \lrs{0}{s}{r}$ be finitely supported filters.
Define a filter $\Theta:=\theta^\star*\tilde{\theta}$, i.e., $\wh{\Theta}(\xi):=\ol{\wh{\theta}(\xi)}^\tp \wh{\tilde{\theta}}(\xi)$.
We now state the discrete multiframelet transform using these finitely supported filters. For any $J\in \N$ and any (vector-valued) input signal/data $v_0\in (\dsq)^{1\times r}$, the $J$-level discrete framelet transform using a filter bank $(\{a;b\},\{\tilde{a};\tilde{b}\})_{\Theta}$ is as follows:
\begin{enumerate}
	\item[(S1)] \emph{Decomposition:} Recursively compute the framelet coefficients $v_j,w_j,j=1,\dots,J$ by
\[
v_j:=\tz_{a,\dm} v_{j-1},\quad w_j:=\tz_{b,\dm} v_{j-1}, \qquad j=1,\ldots, J.
\]
	\item[(S2)] \emph{Reconstruction:} Compute $\tilde{v}_J:=v_J*\Theta$ and recursively compute $\tilde{v}_{j-1}, j=J,\ldots,1$ by
\[
\tilde{v}_{j-1}:= \sd_{\tilde{a},\dm} \tilde{v}_j+\sd_{\tilde{b},\dm} w_{j}, \qquad j=J,\ldots,1.
\]
	\item[(S3)] Recover $\mathring{v}_0$ from $\tilde{v}_0$
through the deconvolution  $\mathring{v}_0*\Theta=\tilde{v}_0$.
\end{enumerate}

\subsection{Perfect reconstruction and compactness of discrete multiframelet transforms}

We say that the $J$-level discrete framelet transform employing the filter bank $(\{a;b\},\{\tilde{a};\tilde{b}\})_{\Theta}$ has \emph{the perfect reconstruction property } if any original input signal $v_0$ can be exactly recovered through the above $J$-level discrete framelet reconstruction steps in (S1)--(S3).
%

For $\Theta\in\dlrs{0}{r}{r}$, define the \emph{convolution operator} $C_\Theta: (\dsq)^{1\times r}\rightarrow (\dsq)^{1\times r}$ by
\be\label{conv:op}
C_\Theta (v):=v*\Theta,\qquad \forall v\in(\dsq)^{1\times r}.
\ee
Observe that a $J$-level discrete framelet transform employing the filter bank $(\{a;b\},\{\tilde{a};\tilde{b}\})_\Theta$ has the perfect reconstruction property if and only if
\be \label{pr:0}
\sd_{\tilde{a},\dm}([\tz_{a,\dm} v]*\Theta)+
\sd_{\tilde{b},\dm} (\tz_{b,\dm} v)=v*\Theta
\ee
holds for all $v\in (\dsq)^{1\times r}$ and the convolution operator $C_\Theta$ in \er{conv:op} is bijective.

\begin{lemma}\label{lem:conv}
For $\Theta\in \dlrs{0}{r}{r}$, the mapping $C_\Theta: (\dsq)^{1\times r}\rightarrow (\dsq)^{1\times r}$ is bijective if and only if $\Theta$ is strongly invertible, that is, $\wh{\Theta}^{-1}$ is an $r\times r$ matrix of $2\pi\dZ$-periodic trigonometric polynomials.
\end{lemma}

\bp Suppose that $C_\Theta$ is bijective, but $\Theta$ is not strongly invertible. This means that $\det(\wh{\Theta})$ is not a non-trivial monomial. Here a non-trivial monomial is of the form $ce^{ik\cdot\xi}$ for some $c\in\C\setminus\{0\}$ and $k\in\dZ$. Thus, $\det(\wh{\Theta}(\xi_0))=0$ for some $\xi_0\in\dC$. We start with the case $r=1$. In this case, we have
$0=\wh{\Theta}(\xi_0)=\sum_{k\in\dZ}\Theta(k)e^{-ik\cdot\xi_0}$.
Define $v\in  \dsq$ by
\be\label{v0}
v(k)=e^{ik\cdot\xi_0},\quad k\in\dZ.
\ee
By definition, we have
$$(v*\Theta)(n)=\sum_{k\in\dZ}v(k)\Theta(n-k)=e^{in\cdot\xi_0}\sum_{k\in\dZ}e^{-i(n-k)\xi_0}\Theta(n-k)=e^{in\cdot\xi_0}\wh{\Theta}(\xi_0)=0,\quad\forall n\in\dZ,$$
which contradicts the injectivity of $C_\Theta$.
For $r>1$, as $\det(\wh{\Theta}(\xi_0))=0$, we can find an invertible $r\times r$ matrix $Q$ such that all elements in the first row of $Q\wh{\Theta}(\xi_0)$ are zero. Let $v\in  \dsq$ be defined as in \er{v0}. Define $u:=(v,0,\dots,0)Q\in (\dsq)^{1\times r}$. It follows that $u*\Theta=0$, which again contradicts the assumption that $C_\Theta$ is injective. Therefore, $\Theta$ must be strongly invertible.

Conversely, if $\Theta$ is strongly invertible, then $\Theta^{-1}\in\dlrs{0}{r}{r}$ with $\wh{\Theta^{-1}}:=[\wh{\Theta}]^{-1}$. Consequently, we have
$v=(v*\Theta)*\Theta^{-1}=(v*\Theta^{-1})*\Theta$ for $v\in (\dsq)^{1\times r}$.
Hence, $C_\Theta$ is bijective.
\ep

We now characterize the perfect reconstruction property of a $J$-level discrete framelet transform.

\begin{theorem}\label{thm:ffrt:pr}
	Let $a,\tilde a, \theta, \tilde{\theta}\in \dlrs{0}{r}{r}$ and
	$b,\tilde{b}\in \dlrs{0}{s}{r}$ be finitely supported filters. Define $\Theta:=\theta^\star*\tilde{\theta}$. Then the following statements are equivalent to each other:
	\begin{enumerate}
		\item[(1)] For any $J\in \N$, the $J$-level discrete framelet transform employing the filter bank $(\{a;b\};\{\tilde{a};\tilde{b}\})_{\Theta}$ has the perfect reconstruction property.
		
		\item[(2)]  Both filters $\theta$ and $\tilde{\theta}$ are strongly invertible and $(\{a;b\},\{\tilde{a};\tilde{b}\})_{\Theta}$ is an OEP-based dual $\dm$-framelet filter bank satisfying \eqref{dffb}.
	\end{enumerate}
\end{theorem}

\bp (1) $\Rightarrow$ (2). Suppose that item (1) holds. By Lemma~\ref{lem:conv},  $\Theta$ is strongly invertible, and thus implies that both $\theta$ and $\tilde{\theta}$ are strongly invertible. On the other hand, observe that
\begin{align}&\wh{\sd_{a,\dm} v}(\xi)=|\det(\dm)|^{1/2} \wh{v}(\dm^{\tp} \xi) \wh{a}(\xi),\label{sd:fourier}\\
&\wh{\tz_{a,\dm} v}(\xi)=|\det(\dm)|^{-1/2}
\sum_{\omega\in\Omega_{\dm}} \wh{v}(\dm^{-\tp}\xi+2\pi\omega)
\ol{\wh{a}(\dm^{-\tp}\xi+2\pi\omega)}^\tp,\label{tz:fourier}\end{align}
for all $v\in\dlrs{0}{r}{r}$. Therefore, \er{pr:0} yields that for all $v\in \dlrs{0}{r}{r}$,
\be \label{pr:2}
\sum_{\omega\in\Omega_{\dm}}
\wh{v}(\xi+2\pi\omega)
\left[\ol{\wh{a}(\xi+2\pi\omega)}^\tp\wh{\Theta}(\dm^{\tp} \xi)
\wh{\tilde{a}}(\xi)+\ol{\wh{b}(\xi+2\pi\omega)}^\tp
\wh{\tilde{b}}(\xi)\right]=\wh{v}(\xi)\wh{\Theta}(\xi).
\ee
Plugging $\wh{v}(\xi)=e^{-i\gamma\cdot\xi}I_r$ with $\gamma\in\Gamma_{\dm} $ into \er{pr:2} and using the same argument as in the proof of \cite[Theorem~2.1]{han13}, we deduce from \er{pr:2} that \er{dffb} must hold. This proves (1) $\Rightarrow$ (2).

(2) $\Rightarrow$ (1). Suppose item (2) holds. Then \er{dffb} implies that \er{pr:0} must hold for all $v\in \dlrs{0}{1}{r}$. Use the locality of the subdivision and transition operators (see the proof of \cite[Theorem 2.1]{han13}),  we can prove that \er{pr:0} holds for all $v\in(\dsq)^{1\times r}$. Noting that $\Theta$ is strongly invertible, we conclude that the $J$-level discrete framelet transform employing the filter bank $(\{a;b\};\{\tilde{a};\tilde{b}\})_{\Theta}$ has the perfect reconstruction property for every $J\in\N$. This proves (2) $\Rightarrow$ (1).
\ep

If both $\theta$ and $\tilde{\theta}$ are strongly invertible, then we see that the following filters are finitely supported:
\be\label{mr:fft:a}
\wh{\mra}(\xi):=\wh{\theta}(\dm^{\tp}\xi)\wh{a}(\xi)\wh{\theta}(\xi)^{-1},\quad \wh{\tilde{\mra}}(\xi):=\wh{\tilde{\theta}}(\dm^{\tp}\xi)\wh{\tilde{a}}(\xi)\wh{\tilde{\theta}}(\xi)^{-1},
\ee
\be\label{mr:fft:b} \wh{\mrb}(\xi):=\wh{b}(\xi)\wh{\theta}(\xi)^{-1},
\quad\wh{\tilde{\mrb}}(\xi):=\wh{\tilde{b}}(\xi)\wh{\tilde{\theta}}(\xi)^{-1}.
\ee
Therefore, instead of using the dual framelet filter bank $(\{a;b\},\{\tilde{a};\tilde{b}\})_{\Theta}$, we can implement a $J$-level discrete framelet transform using the new dual framelet filter bank $(\{\mra;\mrb\},\{\tilde{\mra};\tilde{\mrb}\})_{\td I_r}$ as follows:
\begin{enumerate}
	\item[(S1')] The $J$-level discrete framelet decomposition: recursively compute
\[
\mrv_j:=\tz_{\mra,\dm} \mrv_{j-1},\quad \mrw_j:=\tz_{\mrb,\dm} \mrv_{j-1}, \qquad j=1,\ldots, J,
\]
for an input data $\mrv_0\in(\dsq)^{1\times r}$.
	
\item[(S2')] The $J$-level discrete framelet reconstruction: recursively compute $\tilde{\mrv}_j, j=J,\ldots,1$ by
\[
\tilde{\mrv}_{j-1}:= \sd_{\tilde{\mra},\dm} \tilde{\mrv}_j+\sd_{\tilde{\mrb},\dm} \mrw_{j}, \qquad j=J,\ldots,1.
\]
\end{enumerate}
We see that the deconvolution step disappears with the new filter bank $(\{\mra;\mrb\},\{\tilde{\mra};\tilde{\mrb}\})_{\td I_r}$, which greatly increases the efficiency of the discrete framelet transform.

 The essence of OEP is to replace the original pair $\phi$ and $\tilde{\phi}$ by another desired pair of $\dm$-refinable vector functions $\mrphi$ and $\tilde{\mrphi}$ satisfying
$$
\wh{\mrphi}(\dm^{\tp}\xi)
=\wh{\mra}(\xi)\wh{\mrphi}(\xi),\qquad \wh{\tilde{\mrphi}}(\dm^{\tp}\xi)=
\wh{\tilde{\mra}}(\xi)\wh{\tilde{\mrphi}}(\xi)
$$
where $\mra$ and $\tilde{\mra}$ are defined via \er{mr:fft:a}, so that we can construct a dual framelet $(\{\mrphi;\psi\},\{\tilde{\mrphi};\tilde{\psi}\})$ with the framelet generators $\psi$ and $\tilde{\psi}$ having the highest possible vanishing moments. However, quite often it is too much to expect that both $\theta$ and $\tilde{\theta}$ are strongly invertible. As a consequence, in most cases, we have to deal with the time-consuming deconvolution process in (S3) of the discrete framelet reconstruction scheme. Although the conditions in Theorem~\ref{thm:df} guarantee that the original input signal $v_0$ must be a solution of the deconvolution problem $v_0*\Theta=\tilde{v}_0$, if $\Theta$ is not strongly invertible, the deconvolution problem may have multiple (or even infinitely many) solutions, so that exact recovery cannot be guaranteed. For the case $r=1$, $\Theta$ is strongly invertible if and only if $\wh{\Theta}$ is a non-trivial monomial; i.e., $\wh{\Theta}(\xi)=ce^{-ik\cdot\xi}$ for some $c\in\C\setminus\{0\}$ and $k\in\dZ$. With such a choice of $\Theta$, we lose the main advantage of OEP for improving the orders of vanishing moments of framelet generators. As discussed in \cite{han09,hl19pp}, when $d=1$ and $r>1$, one can always construct a dual $\dm$-framelet through OEP in Theorem~\ref{thm:df}
from any pair of refinable vector functions such that the dual framelet has the highest possible vanishing moments and both $\theta$ and $\tilde{\theta}$ are strongly invertible. These results demonstrate great advantages of multiframelets over scalar framelets.

\subsection{Balancing property of discrete multiframelet transforms}

In many applications, the input data $v$ is given as a scalar sequence, i.e., $v\in\dsq$. However, the input to a multiframelet transform is a vector sequence in $(\sq)^{1\times r}$. Hence, we have to convert a scalar sequence to a vector sequence by using vector conversion. 
Let $\dn$ be a $d\times d$ integer matrix with $|\det(\dn)|=r$, and let $\Gamma_{\dn}$ be a complete set of canonical representatives of the quotient group $\dZ/[\dn\dZ]$ given by
\be\label{enum:dn}
\Gamma_{\dn}:=\{\ka{1},\ldots,\ka{r}\}
\quad \mbox{with}\quad \ka{1}:=0
\quad \mbox{such that $\dZ$ is the disjoint union of}\; \Gamma_{\dn}+\dn\dZ.
\ee
We define \emph{the standard vector conversion operator associated with $\dn$} via
\be\label{vec:con}
[E_{\dn} v](k):=(v(\dn k+\ka{1}),v(\dn k+\ka{2}),\ldots, v(\dn k+\ka{r})),\qquad k\in \dZ,\quad v\in\dsq.
\ee
It is obvious that $E_{\dn}$ is a linear bijective mapping.
For the case $d=1$, we have a natural choice $\dn=r$ and $\Gamma_{\dn}=\{0,1,\ldots,r-1\}$.

Let $\PL_{m-1}$ be the space of all $d$-variate polynomial sequences of degree less than $m$. The sparsity of a multiframelet transform is described by its ability to annihilate framelet coefficients $w_j$ for polynomial input data.
Let $(\{a;b\},\{\tilde{a};\tilde{b}\})_\Theta$ be a dual $\dm$-framelet filter bank and $(\{\mrphi;\psi\},\{\tilde{\mrphi};\tilde{\psi}\})$ be its corresponding dual $\dm$-framelet.
Define $m:=\sr(\tilde{a},\dm)$.
To have sparsity of a multiframelet transform, it is desired that the discrete framelet transform has the following properties:

\begin{enumerate}
\item[(1)] The operator $\tz_{a,\dm}$ is invariant on $E_{\dn} (\PL_{m-1})$, that is,
\be \label{bp:lowpass}
	\tz_{a,\dm} E_{\dn}(\pp)\in E_{\dn}(\PL_{m-1}),\qquad \forall\; \pp\in \PL_{m-1}.
\ee

	\item[(2)] The high-pass filter $b$ has \emph{order $m$ $E_{\dn}$-balanced vanishing moments}, that is,
	\be \label{bp:highpass}
	\tz_{b,\dm} E_{\dn} (\pp)=0, \qquad \forall\; \pp\in \PL_{m-1}.
	\ee
\end{enumerate}

Items (1) and (2) preserve sparsity for all levels when implementing a multi-level discrete framelet transform, because the framelet coefficients $w_j:=\tz_{b,\dm} \tz_{a,\dm}^{j-1} E_{\dn}(\pp)=0$ for all $\pp\in \PL_{m-1}$ and $j\in \N$. We define $\bvmo(b,\dm,\dn):=m$ with $m$ being the largest such integer in \er{bp:highpass}.
A discrete framelet transform or a filter bank $\{a;b\}$ is \emph{order $m$ $E_{\dn}$-balanced} if both \eqref{bp:lowpass} and \eqref{bp:highpass} hold. In particular, we define $\bpo(\{a;b\},\dm,\dn):=m$ with $m$ being the largest such integer satisfying both \eqref{bp:lowpass} and \eqref{bp:highpass}.
Balancing properties for multiwavelets have been studied in \cite{cj00,han09,hanbook,lv98,sel00} and references therein.

Let $a,\tilde{a},\theta,\tilde{\theta}\in\dlrs{0}{r}{r}$ and $b,\tilde{b}\in\dlrs{0}{s}{r}$ such that $(\{a;b\},\{\tilde{a};\tilde{b}\})_{\Theta}$ is an OEP-based dual $\dm$-multiframelet filter bank, where $\Theta=\theta^\star*\tilde{\theta}$. Suppose that $\phi,\tilde{\phi}\in (\dLp{2})^r$ are compactly supported $\dm$-refinable vector functions in $\dLp{2}$ satisfying $\wh{\phi}(\dm^{\tp}\xi)=\wh{a}(\xi)\wh{\phi}(\xi)$
and
$\wh{\tilde{\phi}}(\dm^{\tp}\xi)=\wh{\tilde{a}}(\xi)\wh{\tilde{\phi}}(\xi)$.
Let
\[
\wh{\mrphi}(\xi):=\wh{\theta}(\xi)\wh{\phi}(\xi),\quad
\wh{\psi}(\xi):=\wh{b}(\dm^{-\tp}\xi)\wh{\phi}(\dm^{-\tp}\xi)
\quad \mbox{and}\quad
\wh{\tilde{\mrphi}}(\xi):=\wh{\tilde{\theta}}(\xi)\wh{\tilde{\phi}}(\xi),\quad
\wh{\tilde{\psi}}(\xi):=\wh{\tilde{b}}(\dm^{-\tp}\xi)\wh{\tilde{\phi}}(\dm^{-\tp}\xi).
\]
If $\ol{\wh{\phi}(0)}^\tp \wh{\Theta}(0)\wh{\tilde{\phi}}(0)=1$ and $\wh{\psi}(0)=\wh{\tilde{\psi}}(0)=0$,
then Theorem~\ref{thm:df} tells us that $(\{\mrphi;\psi\},\{\tilde{\mrphi};\tilde{\psi}\})$ is a dual $\dm$-framelet in $\dLp{2}$. If $\sr(\tilde{a},\dm)=m$, we observe that $ \vmo(\psi)\le m$ and $\bpo(\{a;b\},\dm,\dn)\le \bvmo(b,\dm,\dn)\le m$.
If $\bpo(\{a;b\},\dm,\dn)=\bvmo(b,\dm,\dn)
=\vmo(\psi)=m$, then we say that the discrete framelet transform or the filter bank $\{a;b\}$ is \emph{order $m$ $E_{\dn}$-balanced}.
For $r>1$,  $\bpo(\{a,b\},\dm,\dn)<\vmo(\psi)$ often happens. Hence, having high vanishing moments on framelet generators does not guarantee the balancing property and thus  significantly reduces the sparsity of the associated discrete multiframelet transform. How to overcome this shortcoming has been extensively studied in the setting of functions in \cite{cj00,lv98,sel00} and in the setting of discrete framelet transforms in \cite{han09,han10,hanbook}.

The following result on properties of the subdivision and the transition operators that are related to the standard vector conversion operator were investigated, we refer the reader to \cite{han10} for detailed discussions and proofs of the following result.

\begin{theorem}\label{sd:tz:pl}	Let $\dm$ be a $d \times d$ dilation matrix, $s\in\N$ and $r\ge 2$ be positive integers. Let $\dn$ be a $d\times d$ integer matrix with $|\det(\dn)|=r$ and $E_{\dn}$ be the standard vector conversion operator associated with $\dn$ in \eqref{vec:con}.
Define $\Gamma_{\dn}:=\{\ka{1},\ldots,\ka{r}\}$ as in \eqref{enum:dn} and
\be \label{vgu:special}
\wh{\Vgu_{\dn}}(\xi):=\left( e^{i\dn^{-1}\ka{1}\cdot\xi},\ldots, e^{i \dn^{-1}\ka{r}\cdot\xi}\right),\qquad\xi\in\dR.
\ee
Define
$\PL_{m,y}:=\{\pp*y:\pp\in\PL_{m}\}$
for $y\in\dlrs{0}{1}{r}$.
	Then the following statements hold:
	\begin{enumerate}
		 \item[(1)]
$E_{\dn}(\PL_{m})=\PL_{m,y}\subseteq (\PL_{m})^{1\times r}$ with $y\in\dlrs{0}{1}{r}$ if and only if
%
\[
\wh{y}(\xi)=\wh{c}(\xi)\wh{\Vgu_{\dn}}(\xi)+\bo(\|\xi\|^{m+1})\text{ as }\xi\to 0\text{ for some }c\in\dlp{0}\text{ with }\wh{c}(0)\neq 0.
\]
		
		\item[(2)] For $u\in\dlrs{0}{r}{r}$ and $y\in\dlrs{0}{1}{r}$, $\tz_{u,\dm}\PL_{m,y}=\PL_{m,y}$ if and only if
\[
\wh{c}(\xi)\wh{y}(\dm^{\tp}\xi)=\wh{y}(\xi)\ol{\wh{u}(\xi)}^{\tp}+\bo(\|\xi\|^{m+1})\text{ as }\xi\to 0\text{ for some }c\in\dlp{0}\text{ with }\wh{c}(0)\neq 0.
\]
		
		\item[(3)] For $u\in\dlrs{0}{r}{r}$ and $y\in\dlrs{0}{1}{r}$, $\sd_{u,\dm}\PL_{m,y}\subseteq (\PL_{m})^{1\times r}$ if and only if
\[
\wh{y}(\dm^{\tp}\xi)\wh{u}(\xi+2\pi\omega)=\bo(\|\xi\|^{m+1}),\qquad\xi\to 0,\quad \omega\in\Omega_{\dm}\setminus\{0\}.
\]
		
		\item[(4)]For $u,\tilde{u}\in\dlrs{0}{r}{r}$ and $y\in\dlrs{0}{1}{r}$, $\sd_{u,\dm}\tz_{\tilde{u},\dm}(v)=v$ for all $v\in\PL_{m,y}$ if and only if
\[
\wh{y}(\xi)\ol{\wh{\tilde{u}}(\xi)}^{\tp}\wh{u}(\xi+2\pi\omega)=\td(\omega)\wh{y}(\xi)+\bo(\|\xi\|^{m+1}),\qquad\xi\to0,\quad\omega\in\Omega_{\dm}.
\]
\end{enumerate}
\end{theorem}

The following result is known (see \cite[Proposition~3.1, Theorem 4.1]{han10}), which characterizes the balancing property of a discrete multiframelet transform in any dimension.

\begin{theorem}\label{thm:bp}
	Let $\dm$ be a $d \times d$ dilation matrix and $r\ge 2$ be a positive integer.
	Let $a\in \dlrs{0}{r}{r}$ and $b\in \dlrs{0}{s}{r}$ for some $s\in\N$. Let $\dn$ be a $d\times d$ integer matrix with $|\det(\dn)|=r$ and $E_{\dn}$ in \eqref{vec:con}.
Define $\wh{\Vgu_{\dn}}$ as in \er{vgu:special}. Then the following statements hold:
	\begin{enumerate}
		
		\item[(1)] The filter $b$ has order $m$ $E_{\dn}$-balanced vanishing moments satisfying \eqref{bp:highpass} if and only if
\be \label{cond:bvmo} \wh{\Vgu_{\dn}}(\xi)\ol{\wh{b}(\xi)}^\tp
=\bo(\|\xi\|^m),\qquad \xi \to 0.
\ee
		\item[(2)] The filter bank $\{a;b\}$ is order $m$ $E_{\dn}$-balanced satisfying both \eqref{bp:lowpass} and \eqref{bp:highpass} if and only if \eqref{cond:bvmo} holds and
\be \label{cond:ao}		 \wh{\Vgu_{\dn}}(\xi)\ol{\wh{a}(\xi)}^\tp
=	 \wh{c}(\xi)\wh{\Vgu_{\dn}}(\dm^{\tp}\xi)
+\bo(\|\xi\|^m),\quad \xi\to 0\quad\text{for some}\; c\in \dlp{0}\; \text{with}\; \wh{c}(0)\ne 0.
\ee
	\end{enumerate}
\end{theorem}

\section{Proof of Theorem~\ref{thm:normalform} on Normal Form of a Matrix-valued Filter}\label{sec:normalform}

In this section, we shall develop a general normal form of a multivariate matrix filter by proving
Theorem~\ref{thm:normalform}.
Let us first comment on the importance of the normal form theory. First consider the simplest case $d=1$. If a filter $\mra$ takes the standard $(m,n)$-normal form (see item (ii) of Theorem~\ref{thm:normalform}), then \er{mra:11} yields $(1+e^{-i\xi}+\dots+e^{-i(\dm-1)\xi})^m\mid \wh{\mra_{1,1}}(\xi)$ and \er{mra:12} yields $(1-e^{-i\dm\xi})^m\mid \wh{\mra_{1,2}}(\xi)$. So we can factorize $\mra$ into
\[
\wh{\mra}(\xi)=\wh{\Delta_m}(\dm \xi) A(\xi) \wh{\Delta_m}(\xi)^{-1}
\quad \mbox{with} \quad
\wh{\Delta_m}(\xi):=\left[ \begin{matrix} (1-e^{-i\xi})^m &\\
&I_{r-1}\end{matrix}\right]
\]
for a unique matrix $A(\xi)$ of $2\pi$-periodic trigonometric polynomials.
The above factorization technique of taking out the special factor $\wh{\Delta_m}(\xi)$ is the key for constructing univariate multiframelets with high vanishing moments.

However, for $d\ge 2$, there are no corresponding factors for $(1+e^{-i\xi}+\dots+e^{-i(\dm-1)\xi})^m$ and $\wh{\Delta_m}(\xi)$.
But with a bit more effort, the normal form theory allows us to theoretically study and construct multivariate multiwavelets or multiframelets with high vanishing moments from refinable vector functions, in almost the same way as the scalar case (i.e., $r=1$). The new normal form of a matrix-valued filter plays a key role in our study of balanced quasi-tight multiframelets.

To prove Theorem~\ref{thm:normalform}, several auxiliary results are needed. We start with the following result, which is a straightforward generalization of \cite[Lemma 2.3]{han10}.
\begin{lemma}\label{lem:U}
	Let $\wh{v}=(\wh{v_1},\ldots,\wh{v_r})$ and $\wh{u}=(\wh{u_1},\ldots,\wh{u_r})$ be $1\times r$ vectors of functions which are infinitely differentiable at $0$ with $\wh{v}(0)\ne 0$ and $\wh{u}(0)\ne 0$. If $r\ge 2$, then for any positive integer $n\in \N$, there exists a strongly invertible $r\times r$ matrix $\wh{U}$ of $2\pi\dZ$-periodic trigonometric polynomials such that
	 $\wh{u}(\xi)=\wh{v}(\xi)\wh{U}(\xi)+\bo(\|\xi\|^n)$ as $\xi\to 0$.
\end{lemma}

Next, we establish the following lemma on the moment conditions for vectors of smooth functions.

\begin{lemma}\label{lem:moment}
	Let $m\in \N$.
	Let $\wh{v}$ be a $1\times r$ row vector and $\wh{u}$ be an $r\times 1$ column vector such that all the entries of
	$\wh{v}$ and $\wh{u}$ are functions which are infinitely differentiable at the origin such that
	\be \label{vu=1:m}
	 \wh{v}(\xi)\wh{u}(\xi)=1+\bo(\|\xi\|^m),\quad \xi \to 0.
	\ee
	Then for any positive integer $n$, there exists an $1\times r$ vector $\wh{\mathring{v}}$ of $2\pi\dZ$-periodic trigonometric polynomials such that
	\be \label{vu=1:general}
	 \wh{\mathring{v}}(\xi)=\wh{v}(\xi)+\bo(\|\xi\|^m)
	\quad \mbox{and}\quad
	 \wh{\mathring{v}}(\xi)\wh{u}(\xi)=1+\bo(\|\xi\|^n),\quad \xi\to 0.
	\ee
\end{lemma}

\bp The claim trivially holds if $n\le m$, by simply taking $\wh{\mathring{v}}(\xi):=\wh{v}(\xi)+\bo(\|\xi\|^m)$ as $\xi\to 0$.

Consider the case $n>m$. If $r=1$, then $\wh{u}(0)\ne 0$. Take  
$\wh{\mathring{v}}(\xi):=1/\wh{u}(\xi)+\bo(\|\xi\|^n)$, we see that \eqref{vu=1:general} is satisfied. For $r\ge 2$, by Lemma~\ref{lem:U}, there exists a strongly invertible $r\times r$ matrix $\wh{U}$ such that $\wh{\breve{u}}(\xi):=\wh{U}(\xi)\wh{u}(\xi)=(1,0,\ldots,0)^\tp+\bo(\|\xi\|^n)$ as $\xi \to 0$.
We define $\wh{\breve{v}}(\xi)=
(\wh{\breve{v}_1}(\xi),\ldots,\wh{\breve{v}_r}(\xi))
:= \wh{v}(\xi)\wh{U}(\xi)^{-1}$. Then it follows from \eqref{vu=1:m} that
\[
\wh{\breve{v}_1}(\xi)=
\wh{\breve{v}}(\xi)\wh{\breve{u}}(\xi)+\bo(\|\xi\|^n)
=\wh{v}(\xi)\wh{U}(\xi)^{-1} \wh{U}(\xi) \wh{u}(\xi)+\bo(\|\xi\|^n)
=\wh{v}(\xi)\wh{u}(\xi)=1+\bo(\|\xi\|^m), \quad \xi \to 0.
\]
Let $\wh{\mrv}$ be an $1\times r$ vector of $2\pi\dZ$-periodic trigonometric polynomials such that
$$\wh{\mathring{v}}(\xi)=(1, \wh{\breve{v}_2}(\xi),\ldots, \wh{\breve{v}_r}(\xi))\wh{U}(\xi)+\bo(\|\xi\|^n),\qquad\xi\to 0.$$
Then
\[
\wh{\mathring{v}}(\xi) \wh{u}(\xi)=
(1,\wh{\breve{v}_2}(\xi),\ldots, \wh{\breve{v}_r}(\xi)) \wh{U}(\xi) \wh{u}(\xi)+\bo(\|\xi\|^n)
=(1,\wh{\breve{v}_2}(\xi),\ldots, \wh{\breve{v}_r}(\xi)) \wh{\breve{u}}(\xi)+\bo(\|\xi\|^n)
=1+\bo(\|\xi\|^n)
\]
as $\xi \to 0$. By $\wh{\breve{v}_1}(\xi)=1+\bo(\|\xi\|^m)$ as $\xi \to 0$, we deduce that
\[
\wh{\mathring{v}}(\xi)=\wh{\breve{v}}(\xi)\wh{U}(\xi)+\bo(\|\xi\|^m)=\wh{v}(\xi)+\bo(\|\xi\|^m), \quad \xi\to 0.
\]
This completes the proof.
\ep

To prove Theorem~\ref{thm:normalform},
we also need the following result
linking a refinable vector function $\phi$ with the matching filter $\vgu$ for the associated matrix-valued filter of $\phi$.

\begin{lemma}\label{lem:vguphi}
	Let $\dm$ be a dilation matrix and $a\in \dlrs{0}{r}{r}$.
	Let $\phi$ be an $r\times 1$ vector of
	compactly supported distributions satisfying $\wh{\phi}(\dm^{\tp} \xi)=\wh{a}(\xi)\wh{\phi}(\xi)$ with $\wh{\phi}(0)\ne 0$.
	If $a$ has order $m$ sum rules with respect to $\dm$ satisfying \eqref{sr} with a matching filter $\vgu\in \dlrs{0}{1}{r}$ and $\wh{\vgu}(0)\wh{\phi}(0)=1$, then
	\be \label{vguphi=1:m}
	 \wh{\vgu}(\xi)\wh{\phi}(\xi)=1+\bo(\|\xi\|^m),\quad
	\xi\to 0.
	\ee
\end{lemma}

\bp This is a special case of \cite[Proposition~3.2]{han03} and we give a proof here. By our assumption on $a$, using $\wh{\vgu}(\dm^{\tp}\xi ) \wh{a}(\xi)=\wh{\vgu}(\xi)+\bo(\|\xi\|^m)$ as $\xi\to 0$ and $\wh{\phi}(\dm^{\tp}\xi)=\wh{a}(\xi)\wh{\phi}(\xi)$, we deduce that
\be\label{rel:vguphi}
\wh{\vgu}(\dm^{\tp} \xi)\wh{\phi}(\dm^{\tp}\xi)
=\wh{\vgu}(\dm^{\tp} \xi) \wh{a}(\xi)\wh{\phi}(\xi)=
\wh{\vgu}(\xi)\wh{\phi}(\xi)+\bo(\|\xi\|^m), \quad \xi\to 0.
\ee
Define $g(\xi):=\wh{\vgu}(\xi)\wh{\phi}(\xi)$. Then its Taylor polynomial approximation is $g(\xi)=\sum_{j=0}^{m-1} g_j(\xi)+\bo(\|\xi\|^m)$ as $\xi \to 0$, where $g_j(\xi):=\sum_{\mu\in \dNN, |\mu|=j} \frac{\partial^\mu g(0)}{\mu!} \xi^\mu$. Then \eqref{rel:vguphi} is equivalent to $g(0)=1$ and
\[
g_j(\dm^\tp \xi)=g_j(\xi)\qquad \forall\, \xi\in \dR, 1\le j<m.
\]
Since all the eigenvalues of $\dm$ are greater than $1$ in modulus, we can deduce (e.g., see \cite[Proposition~2.1]{han03}) that the above identities force $g_j=0$ for all $1\le j<m$. Hence, \eqref{vguphi=1:m} must hold.
\ep

We now prove the following theorem, which generalizes all results on the standard normal form of a matrix-valued filter in \cite{han03,han09,han10,hanbook,hl19pp,hm03} but under much weaker conditions.

\begin{theorem}\label{thm:normalform:gen}
	Let $\dm$ be a $d\times d$ dilation matrix and $a\in \dlrs{0}{r}{r}$ be a matrix-valued filter.
	Let $\phi$ be an $r\times 1$ vector of
	compactly supported distributions satisfying $\wh{\phi}(\dm^{\tp} \xi)=\wh{a}(\xi)\wh{\phi}(\xi)$ with $\wh{\phi}(0)\ne 0$.
	Suppose that the filter $a$ has order $m$ sum rules with respect to $\dm$ satisfying \eqref{sr} with a matching filter $\vgu\in \dlrs{0}{1}{r}$ such that $\wh{\vgu}(0)\wh{\phi}(0)=1$.
	Let $\wh{\mathring{\vgu}}$ be a $1\times r$ row vector and $\wh{u_\phi}$ be an $r\times 1$ column vector
	such that all the entries of
	$\wh{\mathring{\vgu}}$ and $\wh{u_\phi}$ are functions which are infinitely differentiable at $0$ and
	\be \label{vguphi=1:new}
	 \wh{\mathring{\vgu}}(\xi)\wh{u_\phi}(\xi)=1+\bo(\|\xi\|^m), \qquad \xi \to 0.
	\ee
	If $r\ge 2$, then for any positive integer $n\in \N$, there exists a strongly invertible $r\times r$ matrix $\wh{U}$ of $2\pi\dZ$-periodic trigonometric polynomials such that
	\be \label{normalform:general}
	\wh{\vgu}(\xi) \wh{U}(\xi)^{-1}=\wh{\mathring{\vgu}}(\xi)+\bo(\|\xi\|^m)
	\quad \mbox{and}\quad
	\wh{U}(\xi) \wh{\phi}(\xi)=
	 \wh{u_\phi}(\xi)+\bo(\|\xi\|^n),\qquad \xi\to 0.
	\ee
	Define
	 $\wh{\mathring{\phi}}(\xi):=\wh{U}(\xi) \wh{\phi}(\xi)$
	and $\wh{\mathring{a}}(\xi):=\wh{U}(\dm^{\tp}\xi)\wh{a}(\xi) \wh{U}(\xi)^{-1}$.
	Then the following statements hold:
	\begin{enumerate}
		\item[(i)] The new vector function $\mathring{\phi}$ is a vector of compactly supported distributions satisfying $\wh{\mathring{\phi}}(\dm^{\tp} \xi)=\wh{\mathring{a}}(\xi)\wh{\mathring{\phi}}(\xi)$ for all $\xi\in \R$
		and $\wh{\mathring{\phi}}(\xi)=\wh{u_\phi}(\xi)+\bo(\|\xi\|^n)$ as $\xi \to 0$.

		\item[(ii)] The new finitely supported filter $\mathring{a}$ has order $m$ sum rules with respect to $\dm$ with the matching filter $\mathring{\vgu}$
		satisfying $\wh{\mathring{\vgu}}(0)\wh{\mathring{\phi}}(0)=1$
and \eqref{sr} with $a$ and $\vgu$ being replaced by $\mathring{a}$ and $\mathring{\vgu}$, respectively.
	\end{enumerate}
\end{theorem}

\begin{proof}
We first prove \eqref{normalform:general} for $n \ge m$, from which \eqref{normalform:general} for $n<m$ follows trivially.
By Lemma~\ref{lem:vguphi}, we see that \eqref{vguphi=1:m} holds. Note that $\wh{\phi}$ is smooth at every $\xi\in\dR$, which follows from the Paley-Wiener theorem. Thus by \eqref{vguphi=1:new} and Lemma~\ref{lem:moment}, without loss of generality we may assume that
	\be \label{vguphi=1}
	 \wh{\vgu}(\xi)\wh{\phi}(\xi)=1+\bo(\|\xi\|^n) \quad \mbox{and}
	\quad \wh{\mathring{\vgu}}(\xi)\wh{u_\phi}(\xi)=1+\bo(\|\xi\|^n),\qquad \xi \to 0.
	\ee

	Let $\wh{\breve{\vgu}}(\xi)=(1,0,\ldots,0)$. By Lemma~\ref{lem:U}, there exist strongly invertible $r\times r$ matrices $\wh{U_1}$ and $\wh{U_2}$ of $2\pi\dZ$-periodic trigonometric polynomials such that
	\be \label{vgu:breve}
	\wh{\breve{\vgu}}(\xi)=
	 \wh{\mathring{\vgu}}(\xi)\wh{U_1}(\xi)+\bo(\|\xi\|^{n})
	\quad\mbox{and}\quad
	 \wh{\vgu}(\xi)=\wh{\breve{\vgu}}(\xi) \wh{U_2}(\xi)+\bo(\|\xi\|^{n}),\quad
	\xi\to 0.
	\ee
	Define
	\[
	\wh{\breve{u}_\phi}(\xi):=
	\wh{U_1}(\xi)^{-1} \wh{u_\phi}(\xi),
	\qquad
	 \wh{\breve{\phi}}(\xi):=\wh{U_2}(\xi)\wh{\phi}(\xi),
	\qquad\mbox{and}\qquad
	 \wh{\breve{a}}(\xi):=\wh{U_2}(\dm^{\tp}\xi)\wh{a}(\xi) \wh{U_2}(\xi)^{-1}.
	\]
	Then it is obvious that $\wh{\breve{\phi}}(\dm^{\tp} \xi)=\wh{\breve{a}}(\xi)\wh{\breve{\phi}}(\xi)$ and $\breve{a}$ has order $m$ sum rules with the matching filter $\breve{\vgu}$.
	Write $\breve{u}_\phi=(\breve{u}_1,\ldots, \breve{u}_r)^\tp$.
	Using \eqref{vguphi=1} and the fact $\wh{\breve{\vgu}}(\xi)=(1,0,\ldots,0)$,
	we observe that
	\[
	\wh{\breve{u}_1}(\xi)
	=\wh{\breve{\vgu}}(\xi) \wh{\breve{u}_\phi}(\xi)=
	 \wh{\mathring{\vgu}}(\xi)\wh{U_1}(\xi)
	\wh{U_1}(\xi)^{-1} \wh{u_\phi}(\xi)+\bo(\|\xi\|^{n})
	=\wh{\mathring{\vgu}}(\xi) \wh{u_\phi}(\xi)=1+\bo(\|\xi\|^n),\quad \xi\to 0.
	\]
	Write $\breve{\phi}=(\breve{\phi}_1,\ldots, \breve{\phi}_r)^\tp$. By \eqref{vguphi=1}, we have
	\[
	 \wh{\breve{\phi}_1}(\xi)=\wh{\breve{\vgu}}(\xi)\wh{\breve{\phi}}(\xi)=
	\wh{\vgu}(\xi) \wh{U_2}(\xi)^{-1}
	\wh{U_2}(\xi) \wh{\phi}(\xi)=
	 \wh{\vgu}(\xi)\wh{\phi}(\xi)=1+\bo(\|\xi\|^{n}),\quad \xi \to 0.
	\]
	Choose $2\pi\dZ$-periodic trigonometric polynomials $\wh{w_\ell}, \ell=2,\ldots,r$ such that
	\[
 \wh{w_\ell}(\xi)=\wh{\breve{u}_\ell}(\xi)-	 \wh{\breve{\phi}_\ell}(\xi)+\bo(\|\xi\|^{n}),\quad \xi \to 0, \ell=2, \ldots,r.
	\]
	Define
	\[
	\wh{U_3}(\xi):=\left[ \begin{matrix}
	1 &0 &\cdots &0\\
	\wh{w_2}(\xi) &1 &\cdots &0\\
	\vdots &\vdots &\ddots &\vdots\\
	\wh{w_r}(\xi) &0 &\cdots &1\end{matrix}
	\right].
	\]
	It is not hard to see that $\wh{U_3}$ is strongly invertible by noting $\det \wh{U_3}(\xi)=1$ and
	\be \label{breve:phi}
	 \wh{U_3}(\xi)\wh{\breve{\phi}}(\xi)=
	 \wh{\breve{u}_\phi}(\xi)+\bo(\|\xi\|^n),\qquad \xi \to 0.
	\ee

	Let $\wh{U}(\xi):=\wh{U_1}(\xi)\wh{U_3}(\xi) \wh{U_2}(\xi)$. Clearly $\wh{U}$ is strongly invertible. We are left to verify that all the claims  are satisfied. First, using \eqref{vgu:breve} and $n\ge m$, we have
	\begin{align*}
	\wh{\vgu}(\xi)\wh{U}(\xi)^{-1}
	&=\wh{\vgu}(\xi) \wh{U_2}(\xi)^{-1} \wh{U_3}(\xi)^{-1} \wh{U_1}(\xi)^{-1}
	=\wh{\breve{\vgu}}(\xi) \wh{U_3}(\xi)^{-1}\wh{U_1}(\xi)^{-1}+\bo(\|\xi\|^{n})\\
	&=\wh{\breve{\vgu}}(\xi) \wh{U_1}(\xi)^{-1}+\bo(\|\xi\|^{n})
	 =\wh{\mathring{\vgu}}(\xi)+\bo(\|\xi\|^{n})=\wh{\mathring{\vgu}}(\xi)+\bo(\|\xi\|^{m}),
	\end{align*}
	as $\xi \to 0$, where the first equality of the last line follows from the facts that the first row of $\wh{U_3}(\xi)^{-1}$ is $(1,0,\ldots,0)$ and $\wh{\breve{\vgu}}(\xi)=(1,0,\ldots,0)$.
	Similarly, by $\wh{\breve{\phi}}(\xi)=\wh{U_2}(\xi)\wh{\phi}(\xi)$ and using \eqref{breve:phi}, we have
	\[
	\wh{U}(\xi)\wh{\phi}(\xi)=
	\wh{U_1}(\xi)\wh{U_3}(\xi) \wh{U_2}(\xi)\wh{\phi}(\xi)
	=\wh{U_1}(\xi)\wh{U_3}(\xi) \wh{\breve{\phi}}(\xi)
	=\wh{U_1}(\xi) \wh{\breve{u}_\phi}(\xi)+\bo(\|\xi\|^n)
	 =\wh{u_\phi}(\xi)+\bo(\|\xi\|^n),
	\]
	where the last identity follows from $\wh{\breve{u}_\phi}(\xi)=\wh{U_1}(\xi)^{-1}\wh{u_\phi}(\xi)$.
	This proves \eqref{vguphi=1:new}.
	
	Next, we have
	\[
	\wh{\mrphi}(\dm^{\tp} \xi)=\wh{U}(\dm^{\tp} \xi)\wh{\phi}(\dm^{\tp} \xi)
	=\wh{U}(\dm^{\tp}\xi) \wh{a}(\xi)\wh{\phi}(\xi)=
	\wh{\mathring{a}}(\xi) \wh{\mathring{\phi}}(\xi).
	\]
	By \eqref{vguphi=1:new}, we have
	 $\wh{\mathring{\phi}}(\xi)=\wh{U}(\xi)\wh{\phi}(\xi)=
	\wh{u_\phi}(\xi)+\bo(\|\xi\|^n)$ as $\xi \to 0$. This proves item (i).
	
	Since $\wh{U}$ is strongly invertible, we must have $\mathring{a}\in\dlrs{0}{r}{r}$. For $\omega\in\Omega_{\dm}$, we have
\begin{align*}
	\wh{\mathring{\vgu}}(\dm^{\tp} \xi)& \wh{\mathring{a}}(\xi+2\pi\omega)=
	\wh{\vgu}(\dm^{\tp} \xi) \wh{U}(\dm^{\tp} \xi)^{-1}
	\wh{U}(\dm^{\tp}\xi) \wh{a}(\xi+2\pi\omega) \wh{U}(\xi+2\pi\omega)^{-1}\\
	=&\wh{\vgu}(\dm^{\tp}\xi) \wh{a}(\xi+2\pi\omega) \wh{U}(\xi+2\pi\omega)^{-1}=\td(\omega)\wh{\vgu}(\xi)\wh{U}(\xi)^{-1}+\bo(\|\xi\|^m)=\td(\omega)\wh{\mathring{\vgu}}(\xi)+\bo(\|\xi\|^m)
\end{align*}
	as $\xi\to 0$, which proves item (ii).
\end{proof}

We now prove Theorem~\ref{thm:normalform}, which is a special case of Theorem~\ref{thm:normalform:gen}, but with the additional ``almost orthogonality'' in \eqref{eq:ortho}.

\bp[Proof of Theorem~\ref{thm:normalform}]
Choose a strongly invertible $r\times r$ matrix $\wh{U}$ of $2\pi\dZ$-periodic trigonometric polynomials such that all claims of Theorem~\ref{thm:normalform:gen} hold with $\wh{\mathring{\vgu}}(\xi)=
(1,0,\dots,0)$ and $\wh{u_\phi}(\xi)=(1,0,\dots,0)^\tp$. Then we immediately observe that item (1) holds.

Next, we prove item (2). By Theorem~\ref{thm:normalform:gen}, we see that $\wh{\mrphi}(\dm^{\tp}\xi)=\wh{\mra}(\xi)\wh{\mrphi}(\xi)$ and $\mra$ has order $m$ sum rules with respect to $\dm$ with the matching filter $\mathring{\vgu}$. Moreover,
\[
(1,0,\dots,0)\wh{\mra}(\xi+2\pi\omega)=\td(\omega)(1,0,\dots,0)+\bo(\|\xi\|^m),\quad \xi\to 0,\quad\forall \omega\in\Omega_{\dm}.
\]
It follows that
$\wh{\mathring{a}_{1,1}}(\xi+2\pi\omega)=\bo(\|\xi\|^m)$ as $\xi\to 0$ for all  $\omega\in\Omega_{\dm}\setminus\{0\}$, and
$\wh{\mathring{a}_{1,2}}(\xi+2\pi\omega)=\bo(\|\xi\|^m)$ as $\xi\to 0$ for all $\omega\in\Omega_{\dm}$.
This proves the second identity in \er{mra:11} and \er{mra:12}. On the other hand, by \er{normalform:phi} and  $\wh{\mathring{\phi}}(\dm^{\tp}\xi)=\wh{\mathring{a}}(\xi)\wh{\mathring{\phi}}(\xi)$,
it follows immediately that $\wh{\mathring{a}_{1,1}}(\xi)=1+\bo(\|\xi\|^n)$
and $\wh{\mathring{a}_{2,1}}(\xi)=\bo(\|\xi\|^n)$.
This proves the first identity in \er{mra:11} and \er{mra:21}. Hence item (2) is proved.

Finally, we prove item (3).  By Theorem~\ref{thm:normalform:gen}, there is
a strongly invertible $V\in\lrs{0}{r}{r}$ such that
$$\wh{\vgu}(\xi)\wh{V}(\xi)=\wh{\mathring{\vgu}}(\xi)=(1,0,\dots,0)+\bo(\|\xi\|^{m}),\quad \wh{V}(\xi)^{-1}\wh{\phi}(\xi)=\wh{\mathring{\phi}}(\xi)=(1,0,\dots,0)^{\tp}+\bo(\|\xi\|^{\tilde{n}}),\quad \xi\to 0,$$
where $\tilde{n}=\max(m,n)$. It follows from \er{moment:special} and the above identities that
\be\label{moment:special:2}
(1,0,\dots,0)\wh{V}^{-1}(\xi)=
\frac{\ol{\wh{\phi}(\xi)}^{\tp}}{\|\wh{\phi}(\xi)\|^2}
+\bo(\|\xi\|^{m}),\quad \wh{V}(\xi)(1,0,\dots,0)^{\tp}=\wh{\phi}(\xi)+\bo(\|\xi\|^{\tilde{n}}),\quad \xi\to 0.\ee
For $j=1,\dots,r$, denote $\wh{V_j}$ the $j$-th column of $\wh{V}$. By \eqref{moment:special:2} that $\wh{V_1}(\xi)=\wh{\phi}(\xi)+\bo(\|\xi\|^{\tilde{n}})$ as $\xi\to 0$.

Define $\wh{W_1}(\xi):=\wh{V_1}(\xi)$ and choose $w_1\in \lp{0}$ such that $\wh{w_1}(\xi)=\|\wh{\phi}(\xi)\|^{-2}+\bo(\|\xi\|^{\tilde{n}})$ as $\xi\to 0$. For $j=2,\dots,r$, define $W_j\in\lrs{0}{r}{1}$ and choose $w_{j}\in\lp{0}$ recursively via
\[
\wh{W_j}(\xi)=\wh{V_j}(\xi)-\sum_{l=1}^{j-1}\wh{V_j}(\xi)^{\tp}\ol{\wh{W_l}(\xi)}\wh{w}_l(\xi)\wh{W_l}(\xi),
\qquad
\wh{w_{j}}(\xi)=\|\wh{W_{j}}(\xi)\|^{-2}+\bo(\|\xi\|^{\tilde{n}}),\quad\xi\to 0.
\]
Define $W\in\lrs{0}{r}{r}$ via
\[
\wh{W}(\xi):=[\wh{W_1}(\xi),\wh{W_2}(\xi),\dots,\wh{W_r}(\xi)]=[\wh{\phi}(\xi)+\bo(\|\xi\|^{\tilde{n}}),\wh{W_2}(\xi),\dots,\wh{W_r}(\xi)],\quad\xi\to 0.
\]
By construction, we have $\det(\wh{W})=\det(\wh{V})$ and $\wh{W}$ is strongly invertible. For $j=1,\dots,r$, we have
$$
\wh{W_j}(\xi)=\Big(\wh{V_j}(\xi)-
\sum_{l=1}^{j-1}\wh{V_j}(\xi)^{\tp}
\ol{\wh{W_l}(\xi)}
\wh{W_l}(\xi)\|\wh{W_l}(\xi)\|^{-2}\Big)+\bo(\|\xi\|^{\tilde{n}}),\quad\xi\to 0.
$$
Thus
\[
\ol{\wh{W_j}(\xi)}^{\tp}\wh{W_k}(\xi)=\td(j-k)\|\wh{W_j}\|^2+\bo(\|\xi\|^{\tilde{n}}),\quad\xi\to 0.
\]
It follows that
$$\ol{\wh{W}(\xi)}^{\tp}\wh{W}(\xi)=\DG\left(\|\wh{\phi}(\xi)\|^2, \|\wh{W_2}(\xi)\|^2,\dots,\|\wh{W_r}(\xi)\|^2\right)+\bo(\|\xi\|^{\tilde{n}}),\quad\xi\to 0.$$
By letting $\wh{U}:=\wh{W}^{-1}$, we conclude that \er{eq:ortho} holds. Moreover, it is straightforward to check that \er{moment:special:2} holds with $V$ being replaced by $W$. Thus $U$ satisfies items (1) and (2).

Conversely, suppose that $\wh{U}$ is an $r\times r$ strongly invertible matrix of $2\pi\dZ$-trigonometric polynomials satisfying items (1) and (2),  and \er{eq:ortho} holds with $n\ge m$. Note that $\|\wh{\phi}(0)\|^2\neq 0$, which follows from \er{eq:ortho} and the strong invertibility of $\wh{U}$. Using item~(1), \er{eq:ortho} and $n\ge m$, we have
\begin{align*}
\wh{\vgu}(\xi)&=(1,0,\dots,0)\wh{U}(\xi)+\bo(\|\xi\|^m)
=\|\wh{\phi}(\xi)\|^{-2}(1,0,\dots,0)\ol{\wh{U}(\xi)}^{-\tp}\wh{U}(\xi)^{-1}\wh{U}(\xi)+\bo(\|\xi\|^m)\\
&=\|\wh{\phi}(\xi)\|^{-2}(1,0,\dots,0)\ol{\wh{U}(\xi)}^{-\tp}+\bo(\|\xi\|^m)
=\|\wh{\phi}(\xi)\|^{-2}\ol{\wh{\phi}(\xi)}^{\tp}+\bo(\|\xi\|^m),\quad\xi\to 0,
\end{align*}
which implies \er{moment:special}. This proves item (3). This completes the proof.
\ep

\section{Multivariate Balanced Quasi-tight Multiframelets}
\label{sec:qtf}

In this section, we study OEP-based quasi-tight multiframelets with balancing property and compact discrete multiframelet transforms. We shall prove the main result Theorem~\ref{thm:qtf}, and perform further theoretical investigation on multivariate quasi-tight framelets with high balancing orders.

Let us first recall some notations. For $k\in\dZ$, \emph{the difference operator} $\nabla_k$ is defined to be
\[
\nabla_ku(n):=u(n)-u(n-k),\qquad\forall n\in\dZ,\quad u\in(\dsq)^{t\times r}.
\]
For any multi-index $\beta=(\beta_1,\dots,\beta_d)^{\tp}\in\dN_0$, define
$\nabla^\beta:=\nabla^{\beta_1}_{e_1}\nabla^{\beta_2}_{e_2}\dots\nabla^{\beta_d}_{e_d}$,
	where $\{e_1,\dots,e_d\}$ is the standard basis of $\dR$.
For $u\in \dlrs{0}{r}{r}$, we have
\[
\wh{\nabla^{\beta}u}(\xi)=
\wh{\nabla^{\beta}\td}(\xi)\wh{u}(\xi)=
(1-e^{-i\xi_1})^{\beta_1}(1-e^{-i\xi_2})^{\beta_2}
\cdots(1-e^{-i\xi_d})^{\beta_d}\wh{u}(\xi),\quad\xi=(\xi_1,\dots,\xi_d)^{\tp}\in\dR.
\]
For $x=(x_1,\dots,x_d)$ and $y=(y_1,\dots,y_d)$, we say $x\prec y$ if there exists $l\in\{1,\dots,d\}$ such that $x_j=y_j$ for all $j<l$ and $x_l<y_l$. By $x\preceq y$ we mean that $x\prec y$ or $x=y$.

For $d=1$, recall that a $2\pi$-periodic trigonometric polynomial $\wh{u}$ satisfies $\wh{u}(\xi)=\bo(\|\xi\|^m)$ as $\xi\to 0$ if and only if $(1-e^{-i\xi})^m\mid\wh{u}(\xi)$. When $d\ge 2$, we can no longer separate out a moment factor that plays the role of $(1-e^{-i\xi})^m$. Nevertheless, the following result is known in \cite[Theorem 3.6]{han03} and \cite[Lemma 5]{dhacha}, which characterizes the moment condition for arbitrary dimensions.

\begin{lemma}\label{diff:vm}Let $m\in\N$ and $v\in\dlp{0}$. Then $\wh{v}(\xi)=\bo(\|\xi\|^m)$ as $\xi\to 0$ if and only if
$\wh{v}(\xi)=\sum_{\beta\in\dN_{0;m}}
\wh{\nabla^\beta\td}(\xi)
\wh{u_{\beta}}(\xi)$
for some $u_\beta\in\dlp{0}$ for all $\beta\in\dN_{0;m}$, where
$\dN_{0;m}:=\{\beta\in\dN_0:|\beta|=m\}$.
\end{lemma}

\subsection{Proof of Theorem~\ref{thm:qtf}}

To prove Theorem~\ref{thm:qtf}, we start with the following definition.

\begin{definition} Let $\dm$ be a $d\times d$ dilation matrix and $\dn$ be a $d\times d$ integer matrix with $|\det(\dn)|=r\ge 2$. Let $E_{\dn}$ be the vector conversion operator in \er{vec:con} and $\wh{\Vgu_{\dn}}$ in \er{vgu:special}.
	
	\begin{enumerate}
		\item[(a)]
For $\mra\in\dlrs{0}{r}{r}$,
we say that $\mra$ has \emph{order $m$ $E_{\dn}$-balanced sum rules with respect to $\dm$} if $a$ has order $m$ sum rules with respect to $\dm$ with a matching filter $\mrvgu\in \dlrs{0}{1}{r}$ satisfying
\be\label{bsr:mf}
\wh{\mrvgu}(\xi)=\wh{c}(\xi)\wh{\Vgu_{\dn}}(\xi)+\bo(\|\xi\|^m),\qquad\xi\to 0,\text{ for some } c\in\dlp{0} \text{ satisfying } \wh{c}(0)\neq 0.
\ee

\item[(b)]
A filter $\mrvgu\in \dlrs{0}{1}{r}$  is called \emph{an order $m$ $E_{\dn}$-balanced matching filter} for $\mra\in \dlrs{0}{r}{r}$ if it satisfies \eqref{bsr:mf} and
\be\label{bsr:eq}
\wh{\mrvgu}(\dm^{\tp}\xi)\wh{\mra}(\xi+2\pi\omega)=
\td(\omega)\wh{\mrvgu}(\xi)
+\bo(\|\xi\|^m),\qquad\xi\to 0,\quad \omega\in\Omega_{\dm}.
\ee
We define
$\bsr(\mra,\dm,\dn):=m$ with $m$ being the largest such integer satisfying \er{bsr:mf} and \er{bsr:eq}.

		\item[(c)] For $n\in\N$, we say that $\mra$ is \emph{an order $n$ $E_{\dn}$-balanced refinement filter} associated to an $r\times 1$ vector $\mrphi$ of compactly supported distributions if $\wh{\mrphi}(\dm^{\tp}\xi)=\wh{\mra}(\xi)\wh{\mrphi}(\xi)$ for all $\xi\in\dR$ and
		 \be\label{bref}\wh{\mrphi}(\xi)
=\wh{d}(\xi)\ol{\wh{\Vgu_{\dn}}(\xi)}^{\tp}+
\bo(\|\xi\|^n),\quad\xi\to 0,\text{ for some } d\in\dlp{0} \text{ satisfying } \wh{d}(0)\neq 0.\ee
		
	\end{enumerate}
\end{definition}

We first prove a special case of Theorem~\ref{thm:qtf}, which states that certain balanced filters can be used to construct quasi-tight framelets with high order vanishing moments. This result plays a key role in our proof of Theorem~\ref{thm:qtf} on multivariate quasi-tight framelets.

\begin{theorem}\label{thm:qtf:nf}Let $\dm$ be a $d\times d$ dilation matrix and $\dn$ be a $d\times d$ integer matrix with $|\det(\dn)|=r\ge 2$. Define $E_{\dn}$ and $\wh{\Vgu_{\dn}}$ as in \er{vec:con} and \er{vgu:special}, respectively. Suppose that $\mra\in\lrs{0}{r}{r}$ is an order $m$ $E_{\dn}$-balanced refinement filter associated to an $r\times 1$ vector $\mrphi$ of compactly supported functions in $\dLp{2}$ satisfying \er{bref}, and $\mra$ has order $m$ $E_{\dn}$-balanced sum rules with respect to $\dm$ with an order $m$ $E_{\dn}$-balanced matching filter $\mrvgu\in\dlrs{0}{1}{r}$ satisfying \er{bsr:mf}. If
	 \be\label{ba:vgu:phi}
\begin{split}
\wh{\mrvgu}(\xi)=
\|\wh{\mrphi}(\xi)\|^{-2}
\ol{\wh{\mrphi}(\xi)}^{\tp}
+\bo(\|\xi\|^m)
&=\wh{g}(\xi)\wh{\Vgu_{\dn}}(\xi)+\bo(\|\xi\|^m)\text{ as } \xi\to 0\\
&\qquad \text{ for some }g\in\dlp{0}\text{ with }\wh{g}(0)\neq 0,
\end{split}
\ee
	and  	
\be\label{mrphi:moment}
\|\wh{\mrphi}(\xi)\|^2=1+\bo(\|\xi\|^n),\qquad\xi\to 0,\ee
	for some $n\ge 2m$, then there exist $\mrb\in \lrs{0}{s}{r}$ and $\eps_1,\dots,\eps_s\in\{\pm1\}$ for some $s\in\N$ such that
	
	\begin{enumerate}
		 \item[(i)] $\{\mra;\mrb\}_{\td I_r, (\eps_1,\dots,\eps_s)}$ is a quasi-tight $\dm$-multiframelet filter bank satisfying
\be\label{qtffb}
\ol{\wh{\mra}(\xi)}^{\tp}\wh{\mra}(\xi+2\pi\omega)+\ol{\wh{\mrb}(\xi)}^{\tp}\DG(\eps_1,\dots,\eps_s)\wh{\mrb}(\xi+2\pi\omega)=\td(\omega)I_r,\qquad \xi\in\dR,\quad \omega\in\Omega_{\dm},
\ee
where $\td$ and $\Omega_{\dm}$ are defined as in \er{delta:seq} and \er{omega:dm}, respectively.	
		
		\item[(ii)] $\{\mrphi;\psi\}_{(\eps_1,\dots,\eps_s)}$ is a quasi-tight $\dm$-framelet in $\dLp{2}$ and $\psi$ has order $m$ vanishing moments, where
$\wh{\psi}(\xi):=
\wh{\mrb}(\dm^{-\tp}\xi)\wh{\mrphi}(\dm^{-\tp}\xi)$ for $\xi\in \dR$.
	\end{enumerate}
\end{theorem}

\bp As \er{ba:vgu:phi} holds, by Theorem~\ref{thm:normalform}, there exists a strongly invertible $U\in\dlrs{0}{r}{r}$ such that
\begin{align*}
&\wh{\phi}(\xi):=\wh{U}(\xi)
\wh{\mrphi}(\xi)=(1,0,\dots,0)^{\tp}+\bo(\|\xi\|^n),\qquad\xi\to 0,\\
&\wh{\vgu}(\xi):=\wh{\mrvgu}(\xi)\wh{U}(\xi)^{-1}=(1,0,\dots,0)+\bo(\|\xi\|^m),\qquad\xi\to 0,
\end{align*}
and \er{eq:ortho} holds with $\phi$ being replaced by $\mrphi$. Moreover, by letting $\wh{a}(\xi)=\wh{U}(\dm^{\tp}\xi)\wh{\mra}(\xi)\wh{U}(\xi)^{-1}$, we see that $\wh{a}$ takes the standard $(m,n)$-normal form in item (2) of Theorem~\ref{thm:normalform} with $\mra$ being replaced by $a$.
%
%
Enumerate $\Omega_{\dm}$ as in \er{omega:dm}. Define $\wh{\UU}:=\ol{\wh{U}}^{-\tp}\wh{U}^{-1}$ and
\[
\wh{a_1}(\xi):=\wh{\UU}(\xi)-\ol{\wh{a}(\xi)}^\tp\wh{\UU}(\dm^{\tp}\xi)\wh{a}(\xi)
\quad\mbox{and}\quad
\wh{a_j}(\xi):=-\ol{\wh{a}(\xi)}^\tp\wh{\UU}(\dm^{\tp}\xi)
\wh{a}(\xi+2\pi\om{j}),\quad j=2,\dots,d_{\dm}.
\]
For $j=1$, using \er{eq:ortho} and \er{mrphi:moment}, we have
$$\wh{a_1}(\xi)=\begin{bmatrix}1 & \\
& \wh{\tilde{U}}(\xi)\end{bmatrix}-\ol{\wh{a}(\xi)}^\tp\begin{bmatrix}1 & \\
& \wh{\tilde{U}}(\dm^{\tp} \xi)\end{bmatrix}\wh{a}(\xi)+\bo(\|\xi\|^n)=\begin{bmatrix}p_1(\xi) & p_2(\xi) \\
p_3(\xi)& p_4(\xi)\end{bmatrix},\quad\xi\to 0.$$
where $\wh{\tilde{U}}(\xi)=\DG\left(\|\wh{u_2}(\xi)\|^2,\dots,\|\wh{u_r}(\xi)\|^2\right)$ and $\wh{u_j}$ denotes the $j$-th column of $\wh{U}^{-1}$. Here $p_1,p_2,p_3,p_4$ are $1\times 1$, $1\times (r-1)$, $(r-1)\times 1$ and $(r-1)\times (r-1)$ matrices of $2\pi\dZ$-periodic trigonometric polynomials. Using \er{mra:11}, \er{mra:12} and \er{mra:21} with $\mra$ being replaced by $a$, we deduce that
\begin{align}
&p_1(\xi)=1-\left(|\wh{a_{1,1}}(\xi)|^2+\ol{\wh{a_{2,1}}(\xi)}^{\tp}\wh{\tilde{U}}(\dm^{\tp}\xi)\wh{a_{2,1}}(\xi)\right)+\bo(\|\xi\|^n)=\bo(\|\xi\|^n),\qquad\xi\to 0,\label{a11:p1}\\
&p_2(\xi)=-\ol{\wh{a_{1,1}}(\xi)}\wh{a_{1,2}}(\xi)-\ol{\wh{a_{2,1}}(\xi)}^{\tp}\wh{\tilde{U}}(\dm^{\tp}\xi)\wh{a_{2,2}}(\xi)+\bo(\|\xi\|^n)=\bo(\|\xi\|^m),\qquad\xi\to 0.\label{a11:p2}\\
&p_3(\xi)=\ol{p_2(\xi)}^{\tp}=\bo(\|\xi\|^m),\qquad \xi\to0.\label{a11:p3}
\end{align}
For every $\beta\in\dN_0$ with $|\beta|=m$, define $\Delta_\beta:=\DG(\nabla^\beta\td, \td I_{r-1})
\in\dlrs{0}{r}{r}$.
Using \er{a11:p1},\er{a11:p2},\er{a11:p3}, Lemma~\ref{diff:vm} and $n\ge 2m$, we see that there exist $B_{1,\alpha,\beta}\in\dlrs{0}{r}{r}$ for all $\alpha,\beta\in\dN_{0;m}$
such that
\be\label{a1:fac}
\wh{a_1}(\xi)=\sum_{\alpha,\beta\in\dN_{0;m}}\ol{\wh{\Delta_\alpha}(\xi)}^{\tp}\wh{B_{1,\alpha,\beta}}(\xi)\wh{\Delta_\beta}(\xi),
\ee
where $\dN_{0;m}:=\{\beta\in\dN_0:|\beta|=m\}$.
For $j=2,\dots,d_\dm$, we have
$$\wh{a_j}(\xi)=-\ol{\wh{a}(\xi)}^{\tp}\begin{bmatrix}1 &\\
& \wh{\tilde{U}}(\dm^{\tp}\xi)\end{bmatrix}
\wh{a}(\xi+2\pi\om{j})+\bo(\|\dm^{\tp}\xi\|^n)=\begin{bmatrix}
p_{j,1}(\xi) & p_{j,2}(\xi)\\
p_{j,3}(\xi) & p_{j,4}(\xi)
\end{bmatrix}.$$
Here $p_{j,1},p_{j,2},p_{j,3},p_{j,4}$ are $1\times 1$, $1\times (r-1)$, $(r-1)\times 1$ and $(r-1)\times (r-1)$ matrices of $2\pi\dZ$-periodic trigonometric polynomials.
It follows from the identities in \er{mra:11}, \er{mra:12}, \er{mra:21} with $\mra$ being replaced by $a$ and $n\ge 2m$ that as $\xi\to 0$,
%
\begin{align*}
&p_{j,1}(\xi)=-\left(\ol{\wh{a_{1,1}}(\xi)}\wh{a_{1,1}}(\xi+2\pi \om{j})+\ol{\wh{a_{2,1}}(\xi)}^{\tp}\wh{\tilde{U}}(\dm^{\tp}\xi)\wh{a_{2,1}}(\xi+2\pi\om{j})\right)+\bo(\|\dm^{\tp}\xi\|^n)
=\bo(\|\xi\|^m),\\
&p_{j,2}(\xi)=-\ol{\wh{a_{1,1}}(\xi)}
\wh{a_{1,2}}(\xi+2\pi\om{j})-\ol{\wh{a_{2,1}}(\xi)}^{\tp}\wh{\tilde{U}}(\dm^{\tp}\xi)
P_{2,2}(\xi+2\pi\om{j})+\bo(\|\dm^{\tp}\xi\|^n)=\bo(\|\xi\|^m),\\
&p_{j,3}(\xi)=-\ol{\wh{a_{1,2}}(\xi)}
\wh{a_{1,1}}(\xi+2\pi\om{j})-\ol{\wh{a_{2,2}}(\xi)}^{\tp}\wh{\tilde{U}}(\dm^{\tp}\xi)\wh{a_{2,1}}(\xi+2\pi\om{j})+\bo(\|\dm^{\tp}\xi\|^n) =\bo(\|\xi\|^m),
\end{align*}
and using symmetry and the same argument, we further have
\[
p_{j,1}(\xi+2\pi \om{j})= \bo(\|\xi\|^m) \quad \mbox{and}\quad
p_{j,3}(\xi+2\pi \om{j})= \bo(\|\xi\|^m),\quad \xi\to 0, j=2,\ldots,d_{\dm}.
\]
Hence the above identities and Lemma~\ref{diff:vm} yield
\be\label{aj:fac}
\wh{a_j}(\xi)=\sum_{\alpha,\beta\in\dN_{0;m}}\ol{\wh{\Delta_\alpha}(\xi)}^{\tp}\wh{B_{j,\alpha,\beta}}(\xi)
\wh{\Delta_\beta}(\xi+2\pi\om{j}),
\ee
for some $B_{j,\alpha,\beta}\in\dlrs{0}{r}{r}$ for all $\alpha,\beta\in\dN_{0;m}$ and all $j=2,\dots,d_{\dm}$.
Recall that $P_{a;\dm}(\xi):=[\wh{a}(\xi+\om{1}),\ldots,
\wh{a}(\xi+2\pi \om{d_{\dm}})]$ as in \eqref{Pb}.
It follows from \er{a1:fac} and \er{aj:fac} that
\be \label{fac:maW}
\begin{split}
\cM_{a,\UU}(\xi):=
&\DG\left(\wh{\UU}(\xi+\om{1}),
\ldots,\wh{\UU}(\xi+2\pi\om{d_{\dm}})\right)
-\ol{P_{a;\dm}(\xi)}^\tp
\wh{\UU}(\dm^{\tp} \xi) P_{a;\dm}(\xi)\\
=&\sum_{j=1}^{d_{\dm}}
D_{a_j,\omega_j}(\xi)=
\sum_{j=1}^{d_{\dm}}
\sum_{\alpha,\beta\in\dN_{0;m}}
\ol{D_{\Delta_\alpha,0}(\xi)}^{\tp}
D_{B_{j,\alpha,\beta},\omega_j}(\xi)
D_{\Delta_\beta,0}(\xi),
\end{split}
\ee
where $D_{u,\omega}:=D_{u,\omega;\dm}$ is defined via \er{Duni} for every $u\in\dlrs{0}{r}{r}$ and $\omega\in\Omega_{\dm}$ with the subscript $\dm$ being dropped for simplicity. It follows from \er{DEF} and \er{fac:maW} that
\be\label{pr:coset}\begin{aligned} &d_\dm^{-2}\FF_{r;\dm}(\xi)\cM_{a,\UU}(\xi)\ol{\FF_{r;\dm}(\xi)}^{\tp}\\	 =&d_{\dm}^{-4}\sum_{\alpha,\beta\in\dN_{0;m}}
\left(\FF_{r;\dm}(\xi)\ol{D_{\Delta_\alpha,0}(\xi)}^{\tp}
\ol{\FF_{r;\dm}(\xi)}^{\tp}\right)
\left(\sum_{j=1}^{d_{\dm}}\FF_{r;\dm}(\xi)
D_{B_{j,\alpha,\beta},\omega_j}(\xi)
\ol{\FF_{r;\dm}(\xi)}^{\tp}\right)
\left(\FF_{r;\dm}(\xi)
D_{\Delta_\beta,0}(\xi)
\ol{\FF_{r;\dm}(\xi)}^{\tp}\right)\\ =&\sum_{\alpha,\beta\in\dN_{0;m}}
\ol{E_{\Delta_\alpha,0}(\dm^{\tp}\xi)}^{\tp}
\left(d_\dm^{-1}\sum_{j=1}^{d_{\dm}}
E_{B_{j,\alpha,\beta},\omega_j}(\dm^{\tp}\xi)\right)
E_{\Delta_\beta,0}(\dm^{\tp}\xi),
\end{aligned}\ee
where $E_{u,\omega}:=E_{u,\omega;\dm}$ is defined via \er{Euni} for every $u\in\dlrs{0}{r}{r}$ and $\omega\in\Omega_{\dm}$ by dropping the subscript $\dm$. Define
%
\[
\mathring{E}_{\alpha,\beta}(\xi):=
\frac{d_\dm^{-1}}{2}\sum_{j=1}^{d_{\dm}}
\Big( E_{B_{j,\alpha,\beta},\omega_j}(\xi)+\ol{E_{B_{j,\beta,\alpha},\omega_j}(\xi)}^{\tp}\Big).
\]
It is straightforward to see that $\ol{\mathring{E}_{\alpha,\alpha}}^{\tp}=\mathring{E}_{\alpha,\alpha}$ for all $\alpha\in\dN_{0;m}$. It follows that \er{pr:coset} is equivalent to
\be\label{pr:coset:2}\begin{aligned}
&d_\dm^{-2}\FF_{r;\dm}(\xi)\cM_{a,\UU}(\xi)\ol{\FF_{r;\dm}(\xi)}^{\tp}\\
 =&\sum_{\alpha,\beta\in\dN_{0;m},\alpha\prec\beta}
\Big(
\ol{E_{\Delta_\alpha,0}(\dm^{\tp}\xi)}^{\tp}
\mathring{E}_{\alpha,\beta}(\dm^{\tp}\xi)
E_{\Delta_\beta,0}(\dm^{\tp}\xi)+
\ol{E_{\Delta_\beta,0}(\dm^{\tp}\xi)}^{\tp}
\ol{\mathring{E}_{\alpha,\beta}(\dm^{\tp}\xi)}^{\tp}
E_{\Delta_\alpha,0}(\dm^{\tp}\xi)\Big)\\ &+\sum_{\alpha\in\dN_{0;m}}
\ol{E_{\Delta_\alpha,0}(\dm^{\tp}\xi)}^{\tp}
\mathring{E}_{\alpha,\alpha}(\dm^{\tp}\xi)
E_{\Delta_\alpha,0}(\dm^{\tp}\xi).
\end{aligned}
\ee
For
$\alpha,\beta\in\dN_{0;m}$ and $\alpha\prec\beta$, we take any factorization
$\mathring{E}_{\alpha,\beta}(\xi)=
\ol{E_{\alpha,\beta,1}(\xi)}^{\tp}E_{\alpha,\beta,2}(\xi)$
such that $E_{\alpha,\beta,1}$ and $E_{\alpha,\beta,2}$ are $r\times r$ matrices of $2\pi\dZ$-periodic trigonometric polynomials. By calculation:
\begin{align*} &\sum_{\alpha,\beta\in\dN_{0;m},\alpha\prec\beta}
\ol{\left(E_{\alpha,\beta,1}(\xi)	 E_{\Delta_\alpha,0}(\xi)+
E_{\alpha,\beta,2}(\xi)	 E_{\Delta_\beta,0}(\xi)\right)}^{\tp}
\left(E_{\alpha,\beta,1}(\xi)	 E_{\Delta_\alpha,0}(\xi)+
E_{\alpha,\beta,2}(\xi)	 E_{\Delta_\beta,0}(\xi)\right)\\ =&\sum_{\alpha,\beta\in\dN_{0;m},\alpha\prec\beta}
\Big( \ol{E_{\Delta_\alpha,0}(\xi)}^{\tp}
\mathring{E}_{\alpha,\beta}(\xi)
E_{\Delta_\beta,0}(\xi)+
\ol{E_{\Delta_\beta,0}(\xi)}^{\tp}
\ol{\mathring{E}_{\alpha,\beta}(\xi)}^{\tp}
E_{\Delta_\alpha,0}(\xi)\Big)\\ &+\sum_{\alpha,\beta\in\dN_{0;m},\alpha\prec\beta}
\Big( \ol{E_{\Delta_\alpha,0}(\xi)}^{\tp}
\ol{E_{\alpha,\beta,1}(\xi)}^{\tp}
E_{\alpha,\beta,1}(\xi)
E_{\Delta_\alpha,0}(\xi)+
\ol{E_{\Delta_\beta,0}(\xi)}^{\tp}
\ol{E_{\alpha,\beta,2}(\xi)}^{\tp}
E_{\alpha,\beta,2}(\xi)
E_{\Delta_\beta,0}(\xi) \Big) \\ =&\sum_{\alpha,\beta\in\dN_{0;m},\alpha\prec\beta}
\Big( \ol{E_{\Delta_\alpha,0}(\xi)}^{\tp}
\mathring{E}_{\alpha,\beta}(\xi)
E_{\Delta_\beta,0}(\xi)+
\ol{E_{\Delta_\beta,0}(\xi)}^{\tp}
\ol{\mathring{E}_{\alpha,\beta}(\xi)}^{\tp}
E_{\Delta_\alpha,0}(\xi)\Big)\\ &+\sum_{\alpha\in\dN_{0;m}}
\ol{E_{\Delta_\alpha,0}(\xi)}^{\tp}
E_{\alpha}(\xi)E_{\Delta_\alpha,0}(\xi),
\end{align*}
where $E_\alpha$ is a Hermitian $(d_{\dm}r)\times (d_{\dm}r)$ matrix of $2\pi\dZ$-periodic trigonometric polynomials for every $\alpha\in\dN_{0;m}$.
Define $\eps_{\alpha,\beta,k}:=1$ and $b_{\alpha,\beta,k}\in\dlrs{0}{1}{r}$ for  $k=1,\dots,rd_{\dm}$ via
\[
\wh{b_{\alpha,\beta}}(\xi):=\begin{bmatrix}\wh{b_{\alpha,\beta,1}}(\xi)\\
\vdots\\	 \wh{b_{\alpha,\beta,rd_{\dm}}}(\xi)\end{bmatrix}	 :=E_{\alpha,\beta,1}(\dm^{\tp}\xi)\FF_{r;\dm}(\xi)\begin{bmatrix}
	\wh{\Delta_\alpha}(\xi)\\
	\pmb{0}_{(d_\dm-1)r \times r}
\end{bmatrix}+E_{\alpha,\beta,2}(\dm^{\tp}\xi)\FF_{r;\dm}(\xi)\begin{bmatrix}
	\wh{\Delta_\beta}(\xi)\\
	\pmb{0}_{(d_\dm-1)r \times r}
\end{bmatrix},
\]
for all $\alpha,\beta\in\dN_{0;m}$ with $\alpha\prec\beta$, where $\pmb{0}_{q\times t}$ stands for the $q\times t$ zero matrix. Define $P_{b;\dm}$ via \er{Pb} for a matrix-valued filter $b$. By \er{DEF} and $\ol{\FF_{r;\dm}(\xi)}^{\tp}\FF_{r;\dm}(\xi)=d_{\dm}I_{d_{\dm}r}$, it follows that
\be\label{b:mu:nu:1}
\begin{aligned}
P_{b_{\alpha,\beta};\dm}(\xi)=
&E_{\alpha,\beta,1}(\dm^{\tp}\xi)
\FF_{r;\dm}(\xi)
D_{\Delta_\alpha,0}(\xi)+
E_{\alpha,\beta,2}(\dm^{\tp}\xi)
\FF_{r;\dm}(\xi)
D_{\Delta_\beta,0}(\xi)\\	 =&E_{\alpha,\beta,1}(\dm^{\tp}\xi)
E_{\Delta_\alpha,0}(\dm^{\tp}\xi)
\FF_{r;\dm}(\xi)+
E_{\alpha,\beta,2}(\dm^{\tp}\xi)
E_{\Delta_\beta,0}(\dm^{\tp}\xi)\FF_{r;\dm}(\xi).
\end{aligned}\ee
Similarly, for $\ell\in \{1,2\}$, we define $\eps_{\ell;\alpha,k}:=(-1)^{\ell+1}$
and $b_{\ell;\alpha,k}$
by
\[\wh{b_{\ell;\alpha}}(\xi):=\begin{bmatrix}\wh{b_{\ell;\alpha,1}}(\xi)\\
	\vdots\\ \wh{b_{\ell;\alpha,d_{\dm}r}}(\xi)\end{bmatrix}:=\left(p I_r-(-1)^{\ell} q (\mathring{E}_{\alpha,\alpha}(\dm^{\tp}\xi)-E_\alpha(\dm^{\tp}\xi))\right)\FF_{r;\dm}(\xi)\begin{bmatrix}
	\wh{\Delta_\alpha}(\xi)\\
	\pmb{0}_{(d_\dm-1)r \times r}
\end{bmatrix},
\]
%
for $\alpha\in\dN_{0;m}$ and $k=1,\dots,d_{\dm}r$, where $p,q\in\R$ satisfy $p+q=\frac{1}{4}$. We conclude that
\be \label{b:1:mu}
P_{b_{\ell;\alpha};\dm}(\xi)=\left(p I_r-(-1)^\ell q (\mathring{E}_{\alpha,\alpha}(\dm^{\tp}\xi)-
E_\alpha(\dm^{\tp}\xi))\right)
E_{\Delta_\alpha,0}(\dm^{\tp}\xi)\FF_{r;\dm}(\xi),
\ee
%
for $\ell\in \{1,2\}$.
Define
\begin{align*}
\{(b_{\ell},\eps_{\ell}):{\ell}=1,\dots,s\}:=&\{(b_{\alpha,\beta,k},\eps_{\alpha,\beta,k}):\alpha,\beta\in\dN_{0;m}\text{ with }\alpha\prec\beta, k=1,\dots,d_{\dm}r\}\\	 &\cup\{(b_{\ell;\alpha,k},\eps_{\ell;\alpha,k}):\alpha\in\dN_{0;m}, k=1,\dots,d_{\dm}r,\ell=1,2\},
\end{align*}
and let $b:=[b_1^{\tp},\dots,b_s^{\tp}]^{\tp}$. We claim that $\{a;b\}_{\UU;(\eps_1,\dots,\eps_s)}$ is an OEP-based quasi-tight $\dm$-framelet filter bank. Indeed, by \er{pr:coset:2}, \er{b:mu:nu:1} and \er{b:1:mu}, we have
\begin{align*}
&\ol{P_{b;\dm}(\xi)}^{\tp}\DG(\eps_1,\dots,\eps_{s})P_{b;\dm}(\xi)\\
=&\sum_{\alpha,\beta\in\dN_{0;m},\alpha\prec\beta}
\ol{P_{b_{\alpha,\beta};\dm}(\xi)}^{\tp}
P_{b_{\alpha,\beta};\dm}(\xi)+
\sum_{\alpha\in\dN_{0;m}}
\Big(
\ol{P_{b_{1;\alpha};\dm}(\xi)}^{\tp}
P_{b_{1;\alpha};\dm}(\xi)-\ol{P_{b_{2;\alpha};\dm}(\xi)}^{\tp}P_{b_{2;\alpha};\dm}(\xi)\Big)\\	 =&\ol{\FF_{r;\dm}(\xi)}^{\tp}
\left(\sum_{\alpha,\beta\in\dN_{0;m},\alpha\prec\beta}
\Big(\ol{E_{\Delta_\alpha,0}(\dm^{\tp}\xi)}^{\tp}
\mathring{E}_{\alpha,\beta}(\dm^{\tp}\xi)
E_{\Delta_\beta,0}(\dm^{\tp}\xi)+
\ol{E_{\Delta_\beta,0}(\dm^{\tp}\xi)}^{\tp}
\ol{\mathring{E}_{\alpha,\beta}(\dm^{\tp}\xi)}^{\tp}
E_{\Delta_\alpha,0}(\dm^{\tp}\xi)\Big)\right.\\
 &\qquad\qquad \left.+\sum_{\alpha\in\dN_{0;m}}
 \ol{E_{\Delta_\alpha,0}(\dm^{\tp}\xi)}^{\tp}
 \mathring{E}_{\alpha,\alpha}(\dm^{\tp}\xi)E_{\Delta_\alpha,0}(\dm^{\tp}\xi)\right)\FF_{r;\dm}(\xi)\\	 =&\ol{\FF_{r;\dm}(\xi)}^{\tp}\left(d_\dm^{-2}\FF_{r;\dm}(\xi)\cM_{a,\UU}(\xi)\ol{\FF_{r;\dm}(\xi)}^{\tp}\right)\FF_{r;\dm}(\xi)\\
	=&\cM_{a,\UU}(\xi).
\end{align*}
This proves the claim. Define $\mrb\in\dlrs{0}{s}{r}$ by $\wh{\mrb}=\wh{b}\wh{U}$. The above identity is equivalent to saying that  $\{\mra;\mrb\}_{\td I_r;(\eps_1,\dots,\eps_s)}$ is a quasi-tight $\dm$-framelet filter bank. This proves item (i).

By definition in item (ii), $\wh{\psi}(\dm^\tp\xi)=\wh{\mrb}(\xi)\wh{\mrphi}(\xi)$.
Note that $\ol{\wh{\phi}(0)}^{\tp}\wh{\UU}(0)\wh{\phi}(0)=\|\wh{\mrphi}(0)\|^2=1$. Thus by Theorem~\ref{thm:df}, $\{\mrphi;\psi\}_{(\eps_1,\dots,\eps_s)}$ is a quasi-tight $\dm$-framelet in $\dLp{2}$. Moreover, by $n\ge m$, we have
\[
\begin{bmatrix}
\wh{\Delta_\alpha}(\xi)\\
\pmb{0}_{(d_\dm-1)r \times r}
\end{bmatrix}\wh{\phi}(\xi)=\begin{bmatrix}
\wh{\Delta_\alpha}(\xi)\\
\pmb{0}_{(d_\dm-1)r \times r}
\end{bmatrix}[1,\pmb{0}_{1\times {(r-1)}}]^{\tp}+\bo(\|\xi\|^n)=\bo(\|\xi\|^m),\qquad\xi\to 0,\quad\forall\alpha\in\dN_{0;m}.
\]
Thus it follows that
$\wh{\psi}(\dm^{\tp}\xi)=\wh{\mrb}(\xi)\wh{\mrphi}(\xi)=\wh{b}(\xi)\wh{\phi}(\xi)=\bo(\|\xi\|^m)$ as $\xi\to 0$.
This proves item (ii).
\ep

We are now ready to prove the main result Theorem~\ref{thm:qtf} on multivariate quasi-tight framelets.

\bp[Proof of Theorem~\ref{thm:qtf}]By Theorem~\ref{thm:normalform:gen}, there exists a strongly invertible $\theta\in\dlrs{0}{r}{r}$ such that
\be\label{mr:vgu}\wh{\mrvgu}(\xi):=\wh{\vgu}(\xi)\wh{\theta}(\xi)^{-1}=\frac{1}{\sqrt{r}}\wh{\Vgu_{\dn}}(\xi)+\bo(\|\xi\|^m),\qquad\xi\to 0,\ee
\be\label{mr:phi}\wh{\mrphi}(\xi):=\wh{\theta}(\xi)\wh{\phi}(\xi)=\frac{1}{\sqrt{r}}\ol{\wh{\Vgu_{\dn}}(\xi)}^{\tp}+\bo(\|\xi\|^n),\qquad\xi\to 0,\ee
for some $n\ge 2m$, where $\wh{\Vgu_{\dn}}$ is defined in \er{vgu:special}. In fact, the proof below works as long as \er{ba:vgu:phi} and \er{mrphi:moment} hold with $n\ge 2m$.

By the choice of $\theta$, item (2) trivially holds. Let $\wh{\mra}(\xi):=\wh{\theta}(\dm^{\tp}\xi)\wh{a}(\xi)\wh{\theta}(\xi)^{-1}$. Then $\mra$ is an order $m$ $E_{\dn}$-balanced refinement filter associated to the refinement vector function $\mrphi$, with the order $m$ $E_{\dn}$-balanced matching filter $\mrvgu\in\dlrs{0}{1}{r}$ satisfying \er{bsr:mf}. Moreover, \er{ba:vgu:phi} and \er{mrphi:moment} hold. Thus by Theorem~\ref{thm:qtf:nf}, there exist $b\in\dlrs{0}{s}{r}$ and $\eps_1,\dots,\eps_s\in\{\pm 1\}$ such that items (1) and (3) hold.

On the other hand, using \er{mr:vgu} and \er{mr:phi}, we have
\[
\wh{\Vgu_{\dn}}(\xi)\ol{\wh{\mra}(\xi)}^{\tp}=\sqrt{r}\ol{\wh{\mrphi}(\xi)}^{\tp}\ol{\wh{\mra}(\xi)}^{\tp}+\bo(\|\xi\|^n)=\sqrt{r}\ol{\wh{\mrphi}(\dm^{\tp}\xi)}^{\tp}+\bo(\|\xi\|^n)=\wh{\Vgu_{\dn}}(\dm^{\tp}\xi)+\bo(\|\xi\|^m),\quad \xi\to0,
\]
%
and
%
\[
\wh{\Vgu_{\dn}}(\xi)\ol{\wh{\mrb}(\xi)}^{\tp}=\sqrt{r}\ol{\wh{\mrphi}(\xi)}^{\tp}\ol{\wh{\mrb}(\xi)}^{\tp}+\bo(\|\xi\|^n)=\sqrt{r}\ol{\wh{\psi}(\dm^{\tp}\xi)}^{\tp}+\bo(\|\xi\|^n)=\bo(\|\xi\|^m),\quad \xi\to0.
\]
Hence item (4) follows from Theorem~\ref{thm:bp}. The proof is now complete.
\ep

Though Theorem~\ref{thm:qtf} is for multiplicity $r\ge 2$, one can easily obtain a similar but weaker result for $r=1$. For $r=1$, the notion of balancing property will not come into play since we no longer need the vectorization of scalar data. On the other hand, a scalar filter $\theta\in\dlp{0}$ is strongly invertible if and only if $\wh{\theta}(\xi)=ce^{-ik\cdot\xi}$ for some $c\in\C\setminus\{0\}$ and $k\in\dZ$. Thus it is too much to expect the strong invertibility of $\theta$ when $r=1$ (see \cite{han09} for details about the case $d=1$). By applying the Hermitian matrix decomposition technique as presented in the proof of Theorem~\ref{thm:qtf:nf},  one can still achieve high vanishing moments on framelet generators. We have the following corollary of Theorem~\ref{thm:qtf:nf}.

\begin{cor}\label{cor:qtf:r:1} Let $\dm$ be a $d\times d$ dilation matrix and let $\phi\in \dLp{2}$ be a compactly supported refinable function satisfying $\wh{\phi}(\dm^{\tp} \xi)=\wh{a}(\xi)\wh{\phi}(\xi)$ with $\wh{\phi}(0)\ne 0$, where $a\in\dlp{0}$ has order $m$ sum rules with respect to $\dm$ satisfying \eqref{sr} with a matching filter $\vgu\in \dlp{0}$ such that $\wh{v}(\xi)=1/\wh{\phi}(\xi)+\bo(\|\xi\|^m)$ as $\xi \to 0$. Then there exist $b\in \dlrs{0}{s}{1}$, $\eps_1,\ldots,\eps_s\in \{\pm 1\}$ and $\theta\in \dlp{0}$ such that
	
	\begin{enumerate}
		\item  $\{a;b\}_{\Theta;(\eps_1,\dots,\eps_s)}$ forms an OEP-based quasi-tight $\dm$-framelet filter bank satisfying \eqref{t:fbk}.
		%
		%
		
		\item $\{\mrphi;\psi\}_{(\eps_1,\dots,\eps_s)}$ is a compactly supported quasi-tight $\dm$-framelet in $\dLp{2}$ and $\psi$ has order $m$ vanishing moments, where $\mrphi$ and $\psi$ are defined in \eqref{mphi:mpsi}.
\end{enumerate}
\end{cor}

\subsection{The structure of OEP-based balanced multivariate quasi-tight multiframelets}

In this subsection, we investigate the structure of OEP-based balanced quasi-tight multiframelets. To derive a balanced quasi-tight multiframelet through OEP, the filter $\theta$ in Theorem~\ref{thm:qtf} plays a key role in our investigation. Hence it is important for us to understand the underlying structure that $\theta$ must satisfy. The proof of Theorem~\ref{thm:qtf} reveals some ideas on which strongly invertible $\theta\in\dlrs{0}{r}{r}$ serves as a desired filter for constructing balanced quasi-tight framelets. For simplicity of later discussion, we need the following definition.

\begin{definition}	
Let $\dn$ be a $d\times d$ integer matrix with $|\det(\dn)|=r\ge 2$ and $E_{\dn}$ be the vector conversion operator defined as \er{vec:con}.
Let $\dm$ be a $d\times d$ dilation matrix and $a\in \dlrs{0}{r}{r}$ be a filter associated to a compactly supported refinable vector function $\phi\in (\dLp{2})^r$ satisfying $\wh{\phi}(\dm^{\tp} \xi)=\wh{a}(\xi)\wh{\phi}(\xi)$.  Suppose that the filter $a$ has order $m$ sum rules with respect to $\dm$ satisfying \eqref{sr} with a matching filter $\vgu\in \dlrs{0}{1}{r}$ such that $\wh{\vgu}(0)\wh{\phi}(0)=1$. We say that $\theta\in\dlrs{0}{r}{r}$ is \emph{an order $m$ $E_{\dn}$-balanced moment correction filter} associated to the filter $a$ if $\theta$ is strongly invertible and there exist $b\in \dlrs{0}{s}{r}$, $\eps_1,\ldots,\eps_s\in \{\pm 1\}$ such that all claims in Theorem~\ref{thm:qtf} hold.
\end{definition}

A concrete characterization of balanced moment correction filters is given by the following theorem.

\begin{theorem}\label{thm:momfilter}
Let $\dn$ be a $d\times d$ integer matrix with $|\det(\dn)|=r\ge 2$ and define $E_{\dn}$ and $\wh{\Vgu_{\dn}}$ in \er{vec:con} and \er{vgu:special}, respectively.
Let $\dm$ be a $d \times d$ dilation matrix and $\phi\in (\dLp{2})^r$ be a compactly supported $\dm$-refinable vector function satisfying $\wh{\phi}(\dm^{\tp}\xi)=\wh{a}(\xi)\wh{\phi}(\xi)$ such that
$a\in\dlrs{0}{r}{r}$ has order $m$ sum rules with respect to $\dm$ with a matching filter $\vgu\in\dlrs{0}{1}{r}$ and $\wh{\vgu}(0)\wh{\phi}(0)\neq 0$. Then the following statements hold:

\begin{enumerate}
\item[(i)]
If  $\theta\in\dlrs{0}{r}{r}$ is strongly invertible and if
\er{ba:vgu:phi} and \er{mrphi:moment} hold with $n=2m$ and
\[
\wh{\mra}(\xi):=\wh{\theta}(\dm^{\tp}\xi)\wh{a}(\xi)\wh{\theta}(\xi)^{-1},\quad \wh{\mrvgu}(\xi):=\wh{\vgu}(\xi)\wh{\theta}(\xi)^{-1}\quad \mbox{and}\quad\wh{\mrphi}(\xi):=\wh{\theta}(\xi)\wh{\phi}(\xi),
\]
then $\theta$ is
an order $m$ $E_{\dn}$-balanced moment correction filter associated to $\phi$ and $a$.

\item[(ii)]
If  $\theta\in\dlrs{0}{r}{r}$ is an order $m$ $E_{\dn}$-balanced moment correction filter associated to $\phi$ and $a$, then $\theta$ is strongly invertible and \er{mrphi:moment} must hold with $n=2m$. If in addition
%
\be\label{eigen:a}\text{$1$ is a simple eigenvalue of $\wh{a}(0)$, and $\det(\lambda^{\pm\beta}I_r -\wh{a}(0))\ne 0$ for all $\beta\in\dN_0$ with $0<|\beta|<m$,}
\ee
	where $\lambda=(\lambda_1,\dots,\lambda_{d})$ is the vector of all the eigenvalues of $\dm$, and if
	 \be\label{mrv:sr}
\wh{p}(\dm^{\tp}\xi)\wh{\Vgu_{\dn}}(\dm^{\tp}\xi)\wh{\mra}(\xi)=\wh{p}(\xi)\wh{\Vgu_{\dn}}(\xi)+\bo(\|\xi\|^m)\text{ as }\xi\to 0,
\ee
for some $p\in\dlp{0}$ with $\wh{p}(0)\neq 0$,
then \er{ba:vgu:phi} must hold.
\end{enumerate}
\end{theorem}
	
	\bp Following the lines of the proof of Theorem~\ref{thm:qtf}, if $\theta$ is strongly invertible such that \er{ba:vgu:phi} and \er{mrphi:moment} hold with $n=2m$, then one can obtain $b\in\dlrs{0}{s}{r}$ and $\eps_1,\dots,\eps_s\in\{\pm1\}$ such that all claims of Theorem~\ref{thm:qtf} hold, which implies that $\theta$ must be an order $m$ $E_{\dn}$-balanced moment correction filter associated to $\phi$ and $a$.
This proves item (i).

Conversely, if $\theta$ is an order $m$ $E_{\dn}$-balanced moment correction filter associated to $a$, then there exist  $b\in\dlrs{0}{s}{r}$ and $\eps_1,\dots,\eps_s\in\{\pm1\}$ such that all claims of Theorem~\ref{thm:qtf} hold. In particular,
\be\label{pr:gamma:0}
\ol{\wh{\mra}(\xi)}^{\tp}\wh{\mra}(\xi)+\ol{\wh{\mrb}(\xi)}^{\tp}\DG(\eps_1,\dots,\eps_s)\wh{\mrb}(\xi)=I_r,
\ee
where $\wh{\mrb}(\xi):=\wh{b}(\xi)[\wh{\theta}(\xi)]^{-1}$. By multiplying $\ol{\wh{\mrphi}(\xi)}^{\tp}$ to the left and $\wh{\mrphi}(\xi)$ to the right on both sides of \er{pr:gamma:0}, and using item (1) of Theorem~\ref{thm:qtf}, we deduce that \er{mrphi:moment} holds with $n=2m$.

	By item (4) of Theorem~\ref{thm:qtf}, we see that \er{bp:highpass} holds with $b$ being replaced by $\mrb$ respectively. Consequently, we deduce from \er{sd:fourier}, \er{tz:fourier} and \er{qtffb} that for all $u\in E_{\dn}(\PL_{m-1})$,
\be\label{sd:tz:bp}	 \begin{aligned}\wh{u}(\xi)&=\sum_{j=1}^{d_{\dm}}\wh{u}(\xi+2\pi\om{j})\ol{\wh{\mra}(\xi+2\pi\om{j})}^{\tp}\wh{\mra}(\xi)+\sum_{j=1}^{d_{\dm}}\wh{u}(\xi+2\pi\om{j})\ol{\wh{\mrb}(\xi+2\pi\om{j})}^{\tp}\DG(\eps_1,\dots,\eps_s)\wh{\mrb}(\xi)\\		 &=d_{\dm}^{\frac{1}{2}}\wh{\tz_{\mra,\dm}u}(\dm^{\tp}\xi)\wh{\mra}(\xi)+d_{\dm}^{\frac{1}{2}}\wh{\tz_{\mrb,\dm}u}(\dm^{\tp}\xi)\DG(\eps_1,\dots,\eps_s)\wh{\mrb}(\xi)=\wh{\sd_{\mra,\dm}\tz_{\mra,\dm}u}(\xi).
	\end{aligned}
\ee

Suppose in addition that \er{eigen:a} and \er{mrv:sr} hold. Let $y\in \dlrs{0}{1}{r}$ satisfy
\be\label{y:p}
\wh{y}(\xi)=\wh{p}(\xi)\wh{\Vgu_{\dn}}(\xi)+\bo(\|\xi\|^m)\text{ where }p\text{ is the same as in \er{mrv:sr}}.
\ee
	By item (1) of Theorem~\ref{sd:tz:pl}, we have $E_{\dn}(\PL_{m-1})=\PL_{m-1,y}\subseteq (\PL_{m-1})^{1\times r}$. Thus by item (2) of Theorem~\ref{sd:tz:pl} and \er{sd:tz:bp}, we have
$$
\sd_{\mra,\dm}(\PL_{m-1,y})=\sd_{\mra,\dm}\tz_{\mra,\dm}E_{\dn}(\PL_{m-1})=E_{\dn}(\PL_{m-1})\subseteq (\PL_{m-1})^{1\times r}.
$$
	Hence by item (3) of Theorem~\ref{sd:tz:pl} and \er{mrv:sr}, $\mra$ has order $m$ sum rules with respect to $\dm$, with a matching filter $y\in\dlrs{0}{1}{r}$ satisfying \er{y:p}. On the other hand, since $\mra$ has order $m$ sum rules with a matching filter $\mrvgu$ with $\wh{\mrvgu}:=\wh{\vgu}\wh{\theta}^{-1}$, we have
\[
\wh{\mrvgu}(\dm^\tp \xi)\wh{\mra}(\xi)=\wh{\mrvgu}(\xi)+\bo(\|\xi\|^m),\quad \xi \to 0.
\]
Now the condition in \er{eigen:a} will force $\wh{\mrvgu}(\xi)=\wh{y}(\xi)+\bo(\|\xi\|^m)$
as $\xi \to 0$.

By our assumption in item (ii) on $\theta$, item (4) of Theorem~\ref{thm:qtf} holds. Hence, \eqref{cond:bvmo} and \eqref{cond:ao} of Theorem~\ref{thm:bp} hold with $a=\mra$ and $b=\mrb$. Multiplying  $\wh{\Vgu_{\dn}}(\xi)$ from the left-hand side of \eqref{qtffb} with $\omega=0$, we deduce from \eqref{cond:bvmo} and \eqref{cond:ao} that
\[
\wh{\Vgu_{\dn}}(\xi)=
\wh{\Vgu_{\dn}}(\xi) \ol{\wh{\mra}(\xi)}^\tp \wh{\mra}(\xi)+\bo(\|\xi\|^m)
=\wh{c}(\xi) \wh{\Vgu_{\dn}}(\dm^\tp \xi) \wh{\mra}(\xi)+\bo(\|\xi\|^m)
=\wh{c}(\xi)\frac{\wh{p}(\xi)}{\wh{p}(\dm^\tp \xi)} \wh{\Vgu_{\dn}}(\xi) +\bo(\|\xi\|^m)
\]
as $\xi\to 0$. Since $\wh{p}(0)\ne 0$ and $\wh{\Vgu}(0)\ne 0$, we conclude from the above identity that $\wh{c}(0)=1$.
Since $\wh{\mrphi}(\dm^\tp\xi)=\wh{\mra}(\xi)\wh{\mrphi}(\xi)$
and \eqref{cond:ao} holds with $\wh{c}(0)=1$ and $a=\mra$,
 \er{eigen:a} will force
\[
\wh{\mrphi}(\xi)=\wh{f}(\xi)\ol{\wh{\Vgu_{\dn}}(\xi)}^{\tp}+\bo(\|\xi\|^m),
\quad \xi \to 0
\]
with $\wh{f}(\xi):=\prod_{j=1}^\infty \ol{\wh{c}((\dm^\tp)^{-j} \xi)}$.
Note that $\wh{f}(0)=1$, $\wh{\mrvgu}(\xi)\wh{\mrphi}(\xi)=1+\bo(\|\xi\|^m)$ and $\|\wh{\mrphi}(\xi)\|^2=r\wh{f}(\xi)\ol{\wh{f}(\xi)}+\bo(\|\xi\|^m)$ as $\xi\to 0$. It follows that
$\wh{p}(\xi)=\frac{1}{r\wh{f}(\xi)}+\bo(\|\xi\|^m)$
as $\xi\to 0$.
Thus
\[
\wh{\mrvgu}(\xi)=
[r\wh{f}(\xi)]^{-1}\wh{\Vgu_{\dn}}(\xi)+\bo(\|\xi\|^m)=
\|\wh{\mrphi}(\xi)\|^{-2}\ol{\wh{\mrphi}(\xi)}^{\tp}
+\bo(\|\xi\|^m),\qquad\xi\to 0.
\]
Therefore, \er{ba:vgu:phi} holds. This proves item (ii).
\ep
	
Here we give an example to illustrate such an $E_{\dn}$-balanced moment correction filter $\theta$.

\begin{exmp}\label{expl}
Consider a compactly supported $M_{\sqrt{2}}$-refinable vector function $\phi=(\phi_1,\phi_2)^{\tp}$ given in \cite{goodman94} (see Figure~\ref{fig:psi} for details)
with
its refinement matrix filter $a\in(l_0(\Z^2))^{2\times 2}$ being given by
\be\label{mask:goodman}
\wh{a}(\xi_1, \xi_2):=\frac{1}{4}\begin{bmatrix}2 &1+e^{i\xi_1}+e^{i\xi_2}+e^{i(\xi_1+\xi_2)}\\
		2e^{-i\xi_1} & 0\end{bmatrix}
\quad \mbox{and}\quad
M_{\sqrt{2}}:=\begin{bmatrix}1 & 1\\
		1 & -1\end{bmatrix}.
\ee
	The filter $a$ has order $2$ sum rules with respect to $M_{\sqrt{2}}$, with a matching filter $\vgu\in(l_0(\Z^2))^{1\times 2}$ satisfying
\[
\wh{v}(\xi)=\left(1,1+\frac{i}{2}(\xi_1+\xi_2)\right)+\bo(\|\xi\|^2),\qquad \xi=(\xi_1,\xi_2)\to (0,0).
\]
	Let $\dn:=M_{\sqrt{2}}$. One can obtain an order $2$ $E_{\dn}$-balanced moment correction filter $\theta$ given by
\[
\wh{\theta}(\xi):=\begin{bmatrix}p_1(\xi) & p_2(\xi)\\
		p_3(\xi) & p_4(\xi)\end{bmatrix},\qquad \xi\in\R^2,
\]
where $p_1,p_2,p_3,p_4$ are the following $2\pi\Z^2$-periodic bivariate trigonometric polynomials:
\begin{align*}
&p_1(\xi):=\frac{\sqrt{2}}{272}
\left(542001-3225e^{-2i\xi_1}-7740e^{-i(\xi_1+\xi_2)}
-265735e^{-i\xi_1}+12267e^{_i\xi_2}
-4522e^{i(\xi_1-\xi_2)}-273258e^{i\xi_1}\right),\\
&p_2(\xi):=-\frac{\sqrt{2}}{17}e^{-i(\xi_1+\xi_2)}\left(645e^{-i\xi_1}-646\right),\\
&p_3(\xi):=\frac{\sqrt{2}}{3264}e^{i(\xi_1+\xi_2)}\left[(7740e^{-2i\xi_1}-12267e^{-i\xi_1}+4522)e^{-2i\xi_2}
-1075e^{-2i\xi_1}-89655e^{-i\xi_1}+90873\right.\\
&\qquad\qquad\qquad \qquad\qquad \left.+(3225e^{-3i\xi_1}+265735e^{-2i\xi_1}-544581e^{-i\xi_1}+274763)e^{-i\xi_2}
\right],\\
&p_4(\xi):=\frac{\sqrt{2}}{204}e^{-i\xi_1}(645e^{-i(\xi_1+\xi_2)}-646e^{-i\xi_2}-215).
\end{align*}
Define $\wh{\mrvgu}(\xi):=\wh{\vgu}(\xi)
\wh{\theta}(\xi)^{-1}$ and $\wh{\mrphi}(\xi):=\wh{\theta}(\xi)
\wh{\phi}(\xi)$, we have
$\|\wh{\mrphi}(\xi)\|^2=1+\bo(\|\xi\|^4)$ as $\xi\to(0,0)$ and
\[
\wh{\mrvgu}(\xi)=\ol{\wh{\mrphi}(\xi)}^{\tp}=-\frac{\sqrt{2}}{24}\left(12+429i\xi_1-i\xi_2, 12+435i\xi_1+5\xi_2\right)+\bo(\|\xi\|^2),\qquad\xi=(\xi_1,\xi_2)\to (0,0).
\]
By Theorem~\ref{thm:momfilter}, there exist $b\in (\dlp{0})^{s\times r}$ and $\eps_1,\ldots,\eps_s\in \{\pm 1\}$ such that all the claims in Theorem~\ref{thm:qtf} hold with $m=2$. For simplicity of presentation, we skip details about filters $b\in (\dlp{0})^{s\times r}$.
\end{exmp}

\begin{figure}[htb]
\centering		 \begin{subfigure}[b]{0.3\textwidth}			 \includegraphics[width=\textwidth]
{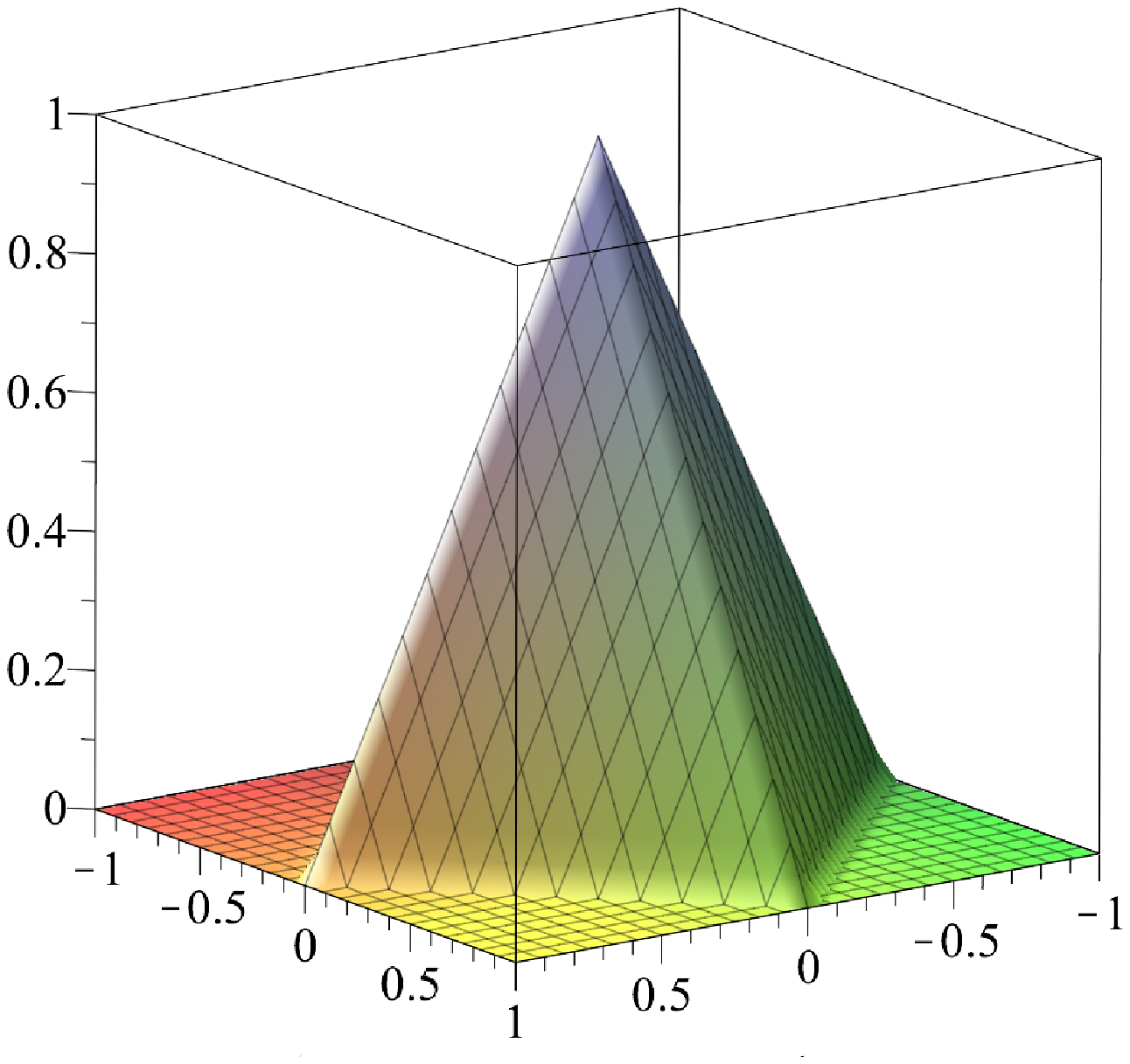}\caption{$\phi_1$} 
\end{subfigure}		 \begin{subfigure}[b]{0.3\textwidth} \includegraphics[width=\textwidth]{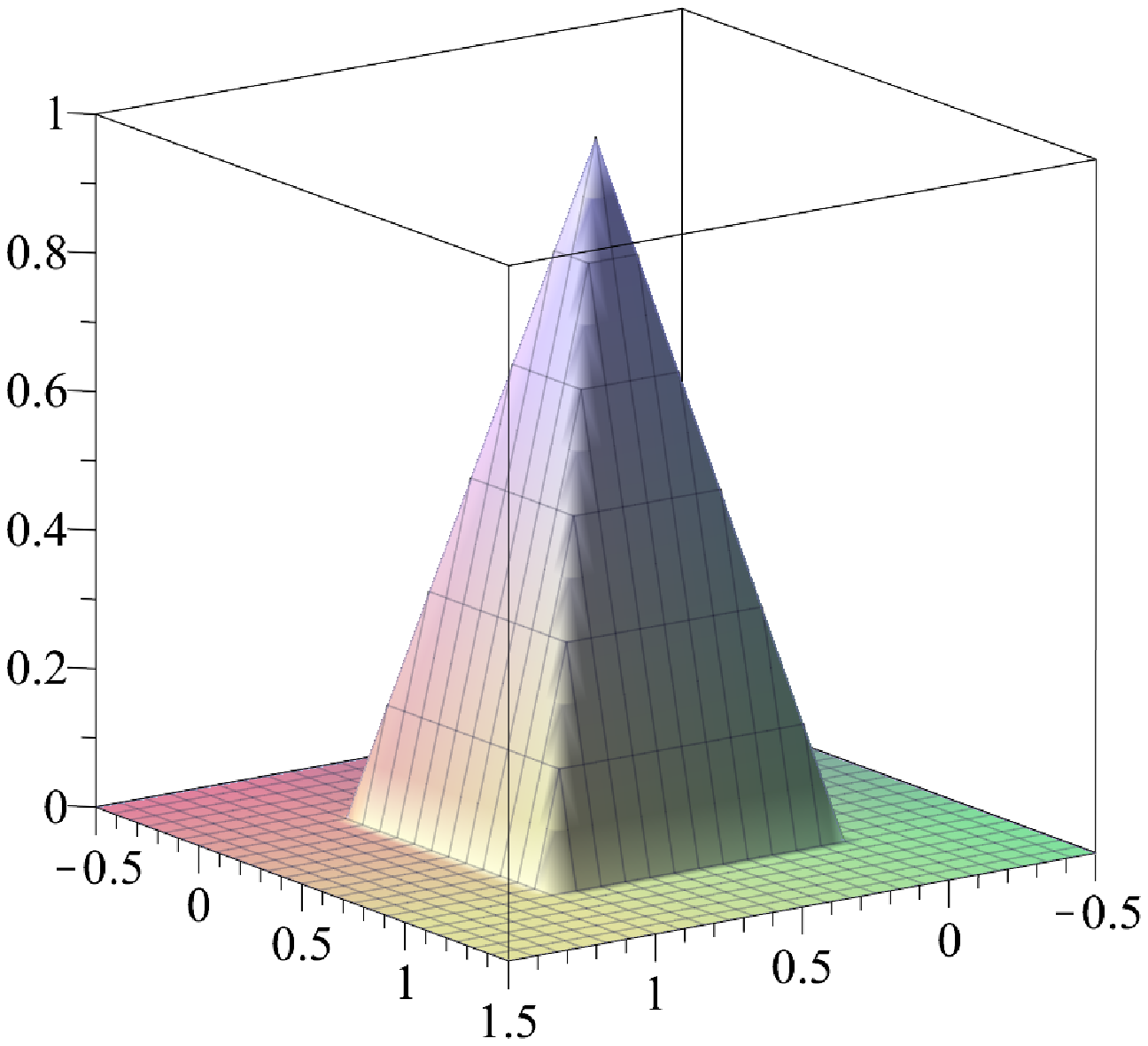}
\caption{$\phi_2$}
\end{subfigure}
\caption{The entries of the $M_{\sqrt{2}}$-refinable vector function $\phi=(\phi_1,\phi_2)^\tp$ in Example~\ref{expl}.}
\label{fig:psi}
\end{figure}

\subsection{Algorithm for constructing balanced multivariate quasi-tight framelets}

The characterization of balanced moment correction filters in Theorem~\ref{thm:momfilter} motivates us to establish an algorithm for constructing quasi-tight multiframelets with high balancing orders.

\begin{lemma}\label{vm:bo:0} Let $r\ge 2$ and $s\in\N$ be positive integers. Let $\dn$ be a $d\times d$ integer matrix with $|\det(\dn)|=r$ and define $E_{\dn}$ as in \er{vec:con}. Then $b\in\dlrs{0}{s}{r}$ has order $m$ $E_{\dn}$-balanced vanishing moments if and only if there exist $q_\beta\in\dlrs{0}{s}{1}$ for all $\beta\in\dN_{0;m}$ such that
\be\label{coset:bbq} \wh{b}(\xi)=\sum_{\beta\in\dN_{0;m}}Q_{q_\beta;\dn}(\xi)\wh{E_{\beta;\dn}}(\xi) \qquad
\mbox{with}\quad
\wh{E_{\beta;\dn}}:=E_{\nabla^\beta\td,0;\dn},
\ee
where $\dN_{0;m}:=\{\beta\in\dN_0:|\beta|=m\}$,
$E_{\nabla^\beta\td,0;\dn}$ is defined as in \er{Euni} with $u$ and $\dm$ being replaced by $\nabla^\beta\td$ and $\dn$, respectively, and $Q_{q_\beta;\dn}$ is defined in \er{Qu} with $u$ and $\dm$ being replaced by $q_\beta$ and $\dn$, respectively.
\end{lemma}

\bp Suppose that $b$ has order $m$ $E_{\dn}$-balanced vanishing moments, i.e., \er{cond:bvmo} holds.
Recall that $\{\ka{1},\ldots,\ka{r}\}=\Gamma_{\dn}$ in \eqref{enum:gamma} with $\dm$ being replaced by $\dn$. For simplicity, we define $u^{[\ka{j}]}:=u^{[\ka{j};\dn]}$ for the $\ka{j}$ coset of $u$ with respect to $\dn$.
Let $\wh{b_j}$ be the $j$-th column of $\wh{b}$ for $j=1,\dots,r$. By Lemma~\ref{diff:vm}, there exist $q_\beta\in\dlrs{0}{s}{1}$ for all $\beta\in\dN_{0;m}$ such that
\be\label{vm:b:r}\begin{aligned}
\sum_{j=1}^r e^{-i\ka{j}\cdot\xi}
\wh{b_j}(\dn^\tp \xi)
=&\sum_{\beta\in\dN_{0;m}}\wh{\nabla^\beta\td}(\xi)\wh{q_\beta}(\xi)\\
=&\sum_{\beta\in\dN_{0;m}}
\left(\sum_{k=1}^r\wh{\nabla^{\beta}\td^{[\ka{k}]}}(\dn^{\tp}\xi)
e^{-i\ka{k}\cdot\xi}\right)
\left(\sum_{l=1}^r\wh{q_\beta^{[\ka{l}]}}(\dn^{\tp}\xi)
e^{-i\ka{l}\cdot\xi}\right)\\ =&\sum_{\beta\in\dN_{0;m}}
\sum_{k=1}^r\sum_{l=1}^r
\wh{\nabla^{\beta}\td^{[\ka{k}]}}(\dn^{\tp}\xi)
\wh{q_\beta^{[\ka{l}]}}(\dn^{\tp}\xi)
e^{-i(\ka{k}+\ka{l})\cdot\xi}.
\end{aligned}
\ee
For every pair of indices $k,l\in\{1,\dots,r\}$, there exist unique $\ka{k,l} \in\Gamma_{\dn}$ and $p_{k,l}\in\dZ$ such that
$\ka{k}+\ka{l}-\ka{k,l}=\dn p_{k,l}$.
Note that
$\wh{u^{[\ka{}]}}(\xi)e^{ip\cdot\xi}
=\wh{u^{[\ka{}+\dn p]}}(\xi)$ for all $\ka{},p\in\dZ$.
It follows that
\be\label{vm:b:r:2}
\begin{aligned}
\wh{\nabla^{\beta}\td^{[\ka{k}]}}(\dn^{\tp}\xi)
&\wh{q_\beta^{[\ka{l}]}}(\dn^{\tp}\xi)e^{-i(\ka{k}+\ka{l})
\cdot\xi}
=\wh{\nabla^{\beta}\td^{[\ka{k}]}}(\dn^{\tp}\xi)
\wh{q_\beta^{[\ka{l}]}}(\dn^{\tp}\xi)e^{-i\ka{k,l}\cdot\xi}
e^{-ip_{k,l}\cdot \dn^{\tp}\xi}\\	 =&\wh{\nabla^{\beta}\td^{[\ka{k}-\dn p_{k,l}]}}(\dn^{\tp}\xi)
\wh{q_\beta^{[\ka{l}]}}(\dn^{\tp}\xi)e^{-i\ka{k,l}\cdot\xi}=
\wh{\nabla^{\beta}\td^{[\ka{k,l}-\ka{k}]}}(\dn^{\tp}\xi)
\wh{q_\beta^{[\ka{l}]}}(\dn^{\tp}\xi)e^{-i\ka{k,l}\cdot\xi}.
\end{aligned}
\ee
Combining \er{vm:b:r} and \er{vm:b:r:2}, we have
\be\label{coset:rep:3}
\sum_{k=1}^r e^{-i\ka{k}\cdot\xi}
\wh{b_k}(\dn^\tp \xi)
=\sum_{k=1}^r\sum_{\beta\in\dN_{0;m}}
\sum_{l=1}^r\wh{\nabla^{\beta}\td^{[\ka{k}-\ka{l}]}}(\dn^{\tp}\xi)
\wh{q_\beta^{[\ka{l}]}}(\dn^{\tp}\xi)e^{-i\ka{k}\cdot\xi}.
\ee
Hence,
$\wh{b_{k}}(\xi)=\sum_{\beta\in\dN_{0;m}}
\sum_{l=1}^r\wh{\nabla^{\beta}\td^{[\ka{k}-\ka{l}]}}(\xi)
\wh{q_\beta^{[\ka{l}]}}(\xi)$ for all $k=1,\dots,r$
and \er{coset:bbq} follows immediately.

Conversely, if \er{coset:bbq} holds, then \er{coset:rep:3} must hold. By working out the above calculation backwards, we obtain \er{vm:b:r}, which precisely means that $b$ has order $m$ $E_{\dn}$-balanced vanishing moments.
\ep

The following result is an immediate consequence of Lemma~\ref{vm:bo:0}.
\begin{prop}\label{prop:b:mfilter}
Let $\dm$ be a $d\times d$ dilation matrix and
$\dn$ be a $d\times d$ integer matrix with $|\det(\dn)|=r\ge 2$ and define $E_{\dn}$ as in \er{vec:con}.
Let $\phi\in (\dLp{2})^r$ be
a compactly supported vector function  satisfying $\wh{\phi}(\dm^{\tp} \xi)=\wh{a}(\xi)\wh{\phi}(\xi)$ with
$a\in \dlrs{0}{r}{r}$.
Suppose $a$ has order $m$ sum rules with respect to $\dm$ satisfying \eqref{sr} with a matching filter $\vgu\in \dlrs{0}{1}{r}$ such that $\wh{\vgu}(0)\wh{\phi}(0)=1$.  If $\theta\in\dlrs{0}{r}{r}$ is an order $m$ $E_{\dn}$-balanced moment correction filter associated to $a$, by letting $\wh{\mra}(\xi)=\wh{\theta}(\dm^{\tp}\xi)\wh{a}(\xi)\wh{\theta}(\xi)^{-1}$ and
\be\label{m:mra:i}
\cM_{\mra}(\xi):=I_{d_{\dm}r}-\ol{P_{\mra;\dm}(\xi)}^{\tp}P_{\mra;\dm}(\xi),
\ee	
then there exist $A_{\alpha,\beta}\in\dlrs{0}{d_{\dm}r}{d_{\dm}r}$ for all $\alpha,\beta\in\dN_{0;m}$ with $\alpha\preceq\beta$ such that
\be\label{a:mu:herm}
\ol{\wh{A_{\alpha,\alpha}}(\xi)}^{\tp}=\wh{A_{\alpha,\alpha}}(\xi)
\ee
and
{\small \be\label{A:mu:nu}
\begin{aligned}
&d_{\dm}^{-2}\FF_{r;\dm}(\xi)\cM_{\mra}(\xi)\ol{\FF_{r;\dm}(\xi)}^{\tp}
=\sum_{\alpha\in\dN_{0;m}}
\ol{E_{E_{\alpha;\dn},0}(\dm^{\tp}\xi)}^{\tp}
\wh{A_{\alpha,\alpha}}(\dm^{\tp}\xi)
E_{E_{\alpha;\dn},0}(\dm^{\tp}\xi)+\\
&
\sum_{\alpha,\beta\in\dN_{0;m},\alpha\prec\beta}
\left[\ol{E_{E_{\alpha;\dn},0}(\dm^{\tp}\xi)}^{\tp}
\wh{A_{\alpha,\beta}}(\dm^{\tp}\xi)
E_{E_{\beta;\dn},0}(\dm^{\tp}\xi)
+\ol{E_{E_{\beta;\dm},0}(\dm^{\tp}\xi)}^{\tp}
\ol{\wh{A_{\alpha,\beta}(\dm^{\tp}\xi)}}^{\tp}
E_{E_{\alpha;\dn},0}(\dm^{\tp}\xi)\right],
\end{aligned}
\ee}
where $F_{r;\dm}$ is defined in \er{Fourier}, $E_{\beta;\dn}$ is defined via \er{coset:bbq}, and $E_{E_{\beta;\dn},0}:=E_{E_{\beta;\dn},0;\dm}$ is defined via \er{Euni} with $u$ and $\omega$ being replaced by $E_{\beta;\dn}$ and $0$ respectively.
\end{prop}

\bp Since $\theta$ is an order $m$ $E_{\dn}$-balanced moment correction filter associated to $\phi$ and $a$, there exist $b\in\dlrs{0}{s}{r}$ and $\eps_1,\dots,\eps_s\in\{\pm1\}$ such that all the claims of Theorem~\ref{thm:qtf} hold. In particular, by letting $\mrb:=b*\theta$, $\{\mra;\mrb\}_{\td I_r, (\eps_1,\dots,\eps_s)}$ is a quasi-tight $\dm$-framelet filter bank satisfying
\[
\cM_{\mra}(\xi)=\ol{P_{\mrb;\dm}(\xi)}^{\tp}\DG(\eps_1,\dots,\eps_s)P_{\mrb;\dm}(\xi).
\]
Moreover, the filter $\mrb$ has order $m$ $E_{\dn}$-balanced vanishing moments. So by Lemma~\ref{vm:bo:0}, there exist $q_\beta\in\dlrs{0}{s}{1}$ for all $\beta\in\dN_{0;m}$ such that \er{coset:bbq} holds with $b$ being replaced by $\mrb$. Define
\[
\wh{\mrb_j}(\xi):=\ol{\wh{\mrb}(\xi)}^{\tp}\DG(\eps_1,\dots,\eps_s)\wh{\mrb}(\xi+2\pi\om{j}),\qquad j=1,\dots,d_{\dm}.
\]
It follows that for $j=1,\dots,d_{\dm}$,
\[
\wh{\mrb_j}(\xi)=\sum_{\alpha,\beta\in\dN_{0;m}}\ol{\wh{E_{\alpha;\dn}}(\xi)}^{\tp}
\ol{Q_{q_\alpha;\dn}(\xi)}^{\tp}\DG(\eps_1,\dots,\eps_s)
Q_{q_\beta;\dn}(\xi+2\pi\omega_j)\wh{E_{\beta;\dn}}(\xi+2\pi\om{j}).
\]

For $j=1,\dots,d_{\dm}$, letting
\[
\wh{q_{\alpha,\beta,j}}(\xi):=\ol{Q_{q_\alpha;\dn}(\xi)}^{\tp}
\DG(\eps_1,\dots,\eps_s)Q_{q_\beta;\dn}(\xi+2\pi\om{j}),\qquad \alpha,\beta\in\dN_{0;m},
\]
we have
\be\label{fac:m:mra}
\begin{aligned}
\cM_{\mra}(\xi)=\sum_{j=1}^{d_{\dm}}D_{\mrb_j,\omega_j;\dm}(\xi)
=\sum_{j=1}^{d_{\dm}}\sum_{\alpha,\beta\in\dN_{0;m}}
\ol{D_{E_{\alpha;\dn},0;\dm}(\xi)}^{\tp}D_{q_{\alpha,\beta,j},\om{j};\dm}(\xi)
D_{E_{\beta;\dn},0;\dm}(\xi).
\end{aligned}
\ee
Note that the decomposition in \er{fac:m:mra} is similar to the one in \er{fac:maW}. Thus by applying the same idea as in the proof of Theorem~\ref{thm:qtf:nf}, one can obtain \er{A:mu:nu}.
\ep

We now provide an algorithm for constructing balanced multivariate quasi-tight framelets. This offers an alternative constructive proof to Theorem~\ref{thm:qtf} on multivariate quasi-tight framelets.

\begin{theorem} \label{bo:vm}	
Let $\dm$ be a $d\times d$ dilation matrix and $\phi\in (\dLp{2})^r$ be a compactly supported vector function satisfying $\wh{\phi}(\dm^{\tp} \xi)=\wh{a}(\xi)\wh{\phi}(\xi)$ with $a\in \dlrs{0}{r}{r}$.
	Suppose that the filter $a$ has order $m$ sum rules with respect to  $\dm$ satisfying \eqref{sr} with a matching filter $\vgu\in \dlrs{0}{1}{r}$ such that $\wh{\vgu}(0)\wh{\phi}(0)=1$.
	If $\dn$ is an $d\times d$ integer matrix with $|\det(\dn)|=r\ge 2$, then one can obtain $b\in \dlrs{0}{s}{r}$, $\eps_1,\ldots,\eps_s\in \{\pm 1\}$ and $\theta\in \dlrs{0}{r}{r}$ such that all claims of Theorem~\ref{thm:qtf} hold by implementing the following steps:
	
	\begin{enumerate}
		\item[(S1)]
Construct a strongly invertible $\theta\in\dlrs{0}{r}{r}$ such that \er{ba:vgu:phi} and \er{mrphi:moment} hold with $n=2m$, where $\wh{\mrvgu}(\xi):=\wh{\vgu}(\xi)
\wh{\theta}^{-1}(\xi)$ and $\wh{\mrphi}(\xi)=\wh{\theta}(\xi)\wh{\phi}(\xi)$.
		
		\item[(S2)]  Define $\wh{\mra}(\xi):=\wh{\theta}(\dm^{\tp}\xi)\wh{a}(\xi)\wh{\theta}(\xi)^{-1}$ and $\cM_{\mra}(\xi)$ as in \er{m:mra:i}. Apply Proposition~\ref{prop:b:mfilter} to find $A_{\alpha,\beta}\in\dlrs{0}{rd_{\dm}}{rd_{\dm}}$ for all $\alpha,\beta\in\dN_{0;m}$ with $\alpha\preceq\beta$ such that \er{a:mu:herm} and \er{A:mu:nu} hold.

		\item[(S3)] For all $\alpha,\beta\in\dN_{0;m}$ with $\alpha\prec\beta$, factorize
$\wh{A_{\alpha,\beta}}(\xi)=\ol{\wh{A_{\alpha,\beta,1}}(\xi)}^{\tp}\wh{A_{\alpha,\beta,2}}(\xi)
$
with $A_{\alpha,\beta,1},A_{\alpha,\beta,2}\in\dlrs{0}{d_{\dm}r}{d_{\dm}r}$.
		Find $B_{\alpha}\in\dlrs{0}{rd_{\dm}}{rd_{\dm}}$ for every $\alpha\in\dN_{0;m}$ such that $\ol{\wh{B_\alpha}(\xi)}^{\tp}=\wh{B_\alpha}(\xi)$ and
{\small
\begin{align*}
&\sum_{\alpha,\beta\in\dN_{0;m},\alpha\prec\beta}
\ol{\left(\wh{A_{\alpha,\beta,1}}(\xi)
E_{E_{\alpha;\dn},0}(\xi)+\wh{A_{\alpha,\beta,2}}(\xi)	 E_{E_{\beta;\dn},0}(\xi)\right)}^{\tp}
\left(\wh{A_{\alpha,\beta,1}}(\xi)	 E_{E_{\alpha;\dn},0}(\xi)+\wh{A_{\alpha,\beta,2}}(\xi) E_{E_{\beta;\dn},0}(\xi)\right)\\ &\qquad\qquad +\sum_{\alpha\in\dN_{0;m}}
\ol{E_{E_{\alpha;\dn},0}(\xi)}^{\tp}
\wh{B_{\alpha}}(\xi)
E_{E_{\alpha;\dn},0}(\xi)=\cN(\xi),
\end{align*}
}
where
$\cN(\dm^{\tp}\xi):=d_{\dm}^{-2}\FF_{r;\dm}(\xi)
\cM_{\mra}(\xi)\ol{\FF_{r;\dm}(\xi)}^{\tp}$ and
$E_{E_{\beta;\dn},0}:=E_{E_{\beta;\dn},0;\dm}$ is defined via \er{Euni} with $u$ and $\omega$ being replaced by $E_{\beta;\dn}$ and $0$, respectively.

\item[(S4)]
Define $\eps_{\alpha,\beta,k}=1$ and $\mrb_{\alpha,\beta,k}\in\dlrs{0}{1}{r}$ for $k=1,\dots,d_{\dm}r$ and $\alpha,\beta\in\dN_{0;m}$ with $\alpha\prec\beta$
via
\[
\wh{\mrb_{\alpha,\beta}}(\xi):=
\begin{bmatrix}\wh{\mrb_{\alpha,\beta,1}}(\xi)\\
			\vdots\\			 \wh{\mrb_{\alpha,\beta,d_{\dm}r}}(\xi)\end{bmatrix}	 :=\wh{A_{\alpha,\beta,1}}(\dm^{\tp}\xi)\FF_{r;\dm}(\xi)\begin{bmatrix}
\wh{E_{\alpha;\dn}}(\xi)\\
\pmb{0}_{(d_\dm-1)r \times r}		 \end{bmatrix}+\wh{A_{\alpha,\beta,2}}(\dm^{\tp}\xi)\FF_{r;\dm}(\xi)\begin{bmatrix}
\wh{E_{\beta;\dn}}(\xi)\\
\pmb{0}_{(d_\dm-1)r \times r}.
\end{bmatrix},
\]
For $\ell\in \{1,2\}$ and $k=1,\dots,d_{\dm}r$, define $\eps_{\ell;\alpha,k}=(-1)^{\ell+1}$ and  $\mrb_{\ell;\alpha,k}\in\dlrs{0}{1}{r}$ by
\[
\wh{\mrb_{\ell;\alpha}}(\xi):=\begin{bmatrix}\wh{\mrb_{\ell;\alpha,1}}(\xi)\\
\vdots\\			 \wh{\mrb_{\ell;\alpha,d_{\dm}r}}(\xi)\end{bmatrix}:=\left(p I_r-(-1)^\ell q (\wh{A_{\alpha,\alpha}}(\dm^{\tp}\xi)-\wh{B_\alpha}(\dm^{\tp}\xi))\right)
\FF_{r;\dm}(\xi)\begin{bmatrix}
\wh{E_{\alpha;\dn}}(\xi)\\
\pmb{0}_{(d_\dm-1)r \times r}
\end{bmatrix}
\]
for $\alpha\in\dN_{0;m}$, where $p,q\in\R$ satisfy $p+q=\frac{1}{4}$.
Define
%
\begin{align*}
\{(\mrb_{\ell},\eps_{\ell}):{\ell}=1,\dots,s\}:=
&\{(\mrb_{\alpha,\beta,k},\eps_{\alpha,\beta,k}):\alpha,\beta\in\dN_{0;m}\text{ with }\alpha\prec\beta, k=1,\dots,d_{\dm}r\}\\			 &\cup\{(\mrb_{\ell;\alpha,k},\eps_{\ell;\alpha,k}):\alpha\in\dN_{0;m}, k=1,\dots,d_{\dm}r,\ell=1,2\}.
\end{align*}
\end{enumerate}
Let $\mrb:=[\mrb_1^{\tp},\dots,\mrb_s^{\tp}]^{\tp}$ and $b=\mrb*\theta$.
Then $\{\mra;\mrb\}_{\td I_r, (\eps_1,\dots,\eps_r)}$ is a finitely supported quasi-tight $\dm$-framelet filter bank such that all the claims of Theorem~\ref{thm:qtf} hold.
\end{theorem}

\bp The existence of $\theta\in\dlrs{0}{r}{r}$ satisfying all the conditions in (S1) is guaranteed by Theorem~\ref{thm:normalform} (e.g. choose $\theta\in (\dlp{0})^{r\times r}$ such that \er{mr:vgu} and \er{mr:phi} hold).  So it is straightforward to see that item (2) of Theorem~\ref{thm:qtf} holds. Moreover, $\theta$ is an order $m$ $E_{\dn}$-balanced moment correction filter associated to $\phi$ and $a$. Thus, (S2) is justified by Proposition~\ref{prop:b:mfilter}. The filters $B_\alpha$ satisfying the identity in (S3) can be obtained by using the same idea as in the proof of Theorem~\ref{thm:qtf:nf}. Now define $\mrb$ as in (S4). By \er{DEF} and the identity $\ol{\FF_{r;\dm}(\xi)}^{\tp}\FF_{r;\dm}(\xi)=d_{\dm}I_{d_{\dm}r}$, we deduce that
\begin{align}
&P_{\mrb_{\alpha,\beta};\dm}(\xi) =\wh{A_{\alpha,\beta,1}}(\dm^{\tp}\xi)
E_{E_{\alpha;\dn},0}(\dm^{\tp}\xi)
\FF_{r;\dm}(\xi)+\wh{A_{\alpha,\beta,2}}(\dm^{\tp}\xi)
E_{E_{\beta;\dn},0}(\dm^{\tp}\xi)\FF_{r;\dm}(\xi),
\label{mrb:mu:nu:1}\\
&P_{\mrb_{\ell;\alpha};\dm}(\xi)
=\left(p I_r-(-1)^\ell q (\wh{A_{\alpha,\alpha}}(\dm^{\tp}\xi)
-\wh{B_\alpha}(\dm^{\tp}\xi))\right)
E_{E_{\alpha;\dn},0}(\dm^{\tp}\xi)
\FF_{r;\dm}(\xi),\label{mrb:1:mu}
\end{align}
for $\ell\in \{1,2\}$.
By \er{mrb:mu:nu:1} and \er{mrb:1:mu}, item (3) of Theorem~\ref{thm:qtf} can be verified by direct calculation.

Define $q_{\alpha,\beta,l},q_{l;\alpha}\in\dlrs{0}{d_{\dm}r}{1}$ for $\ell=1,2$ and for all  $\alpha,\beta\in\dN_{0;m}$ with $\alpha\prec\beta$ such that
%
\begin{align*}
&Q_{q_{\alpha,\beta,\ell};\dn}(\xi):=
\wh{A_{\alpha,\beta,\ell}}(\dm^{\tp}\xi)\FF_{r;\dm}(\xi)\begin{bmatrix}
		I_r\\
		\pmb{0}_{(d_\dm-1)r \times r}\end{bmatrix},\\
&Q_{q_{\ell;\alpha};\dn}(\xi):=\left(p I_r-(-1)^\ell q (\wh{A_{\alpha,\alpha}}(\dm^{\tp}\xi)-\wh{B_\alpha}(\dm^{\tp}\xi))\right)
\FF_{r;\dm}(\xi)\begin{bmatrix}
		I_r\\
		\pmb{0}_{(d_\dm-1)r \times r}
	\end{bmatrix}
\end{align*}
for $\ell\in \{1,2\}$.
We see that
$\wh{\mrb_{\alpha,\beta}}(\xi)=
Q_{q_{\alpha,\beta,1};\dn}(\xi)\wh{E_{\alpha;\dn}}(\xi)
+Q_{q_{\alpha,\beta,2};\dn}(\xi)\wh{E_{\beta;\dn}}(\xi)$
for all $\alpha,\beta\in\dN_{0;m}$ with $\alpha\prec\beta$, and
$\wh{\mrb_{\ell;\alpha}}(\xi)=Q_{q_{\ell;\alpha};\dn}(\xi)\wh{E_{\alpha;\dn}}(\xi)$
for all $\alpha\in\dN_{0;m}$ and $\ell=1,2$. Hence Lemma~\ref{vm:bo:0} implies that $\mrb$ has order $m$ $E_{\dn}$-balanced vanishing moments. Combining this fact with \er{mrphi:moment}, we conclude that items (1) and (4) of Theorem~\ref{thm:qtf} follow right away.
\ep

\section{Summary and Discussion}

In this paper, we provided a self-contained comprehensive investigation on OEP-based compactly supported multivariate quasi-tight multiframelets. From any compactly supported vector refinable function with multiplicity greater than one, we proved that one can always derive a compactly supported quasi-tight multiframelet via OEP such that: (1) its framelet generators have the highest possible orders of vanishing moments; (2) its underlying discrete multiframelet transform is compact and has the highest possible balancing order. We systematically studied the properties of a multi-level discrete framelet transform employing an OPE-based framelet filter bank in Section~\ref{sec:ffrt}. In Section~\ref{sec:normalform}, we further improved a normal form of a matrix-valued filter, which plays a key role in our study of multiframelets. Such a normal form of a matrix-valued filter greatly facilitates the study of refinable vector functions and multiframelets/multiwavelets.
We also provided
structural analysis of OEP-based quasi-tight multiframelets 
in Section~\ref{sec:qtf}.

One natural question is whether it is possible to construct a tight multiframelet satisfying all desired properties with the highest possible order of vanishing moments from any given compactly supported refinable vector function with multiplicity greater than one. Or at least from some special choices of refinable vector functions with multiplicity greater than one. This question is important and closely related to the spectral factorization of multivariate polynomial matrices, which is known to be a challenging problem.
The key for constructing quasi-tight framelets is the construction of an order $m$ $E_\dn$-balanced moment correction filter $\theta$ in Theorem~\ref{thm:momfilter}.
Though the existence of such $E_\dn$-balanced moment correction filters $\theta$ is guaranteed by Theorem~\ref{thm:momfilter}, such strongly invertible $E_\dn$-balanced filters $\theta$ constructed through the normal form of matrix-valued filters
often have long supports. 
It is of practical interest to develop an effective algorithm for constructing an order $m$
$E_\dn$-balanced moment correction filter $\theta$ with the shortest possible support for a given matrix-valued filter $a$. In this paper we didn't address any applications of multivariate quasi-tight framelets. It is important to explore such quasi-tight framelets in some applications for the future developments on quasi-tight framelets.
For applications,
it is also very crucial to construct compactly supported quasi-tight framelets with minimum number of generators and short supports.
We leave these problems for future research.

\end{document}